\documentclass[11pt]{amsart}
\usepackage{geometry}  
\usepackage{amssymb}
\usepackage{amsmath}
\usepackage{amsfonts}
\usepackage[usenames]{color}

\input{epsf}
\usepackage{graphicx}
\usepackage{subfig}
\usepackage{psfrag}
\usepackage{epsfig}
\usepackage{color}

\newtheorem{theorem}{Theorem}[section]

\newtheorem{lemma}[theorem]{Lemma}

\newtheorem{proposition}[theorem]{Proposition}

\newtheorem{corollary}[theorem]{Corollary}

\theoremstyle{definition}
\newtheorem{definition}[theorem]{Definition}
\newtheorem{example}[theorem]{Example}

\theoremstyle{remark}
\newtheorem{remark}[theorem]{Remark}

\numberwithin{equation}{section}

\newcommand{\Z}{\mathbb{Z}}
\newcommand{\R}{\mathbb{R}}
\newcommand{\RR}{\mathbb{R}}
\newcommand{\N}{\mathbb{N}}

\def\im{{\rm im}\,}
\def\rad{{\rm rad}\,}
\def\spn#1{{\rm span}\{#1\}}
\def\Mod#1{{\rm Mod}(#1)}
\def\Diff#1{{\rm Diff}(#1)}

\begin{document} 

\title{Polynomial invariants of pseudo-Anosov maps}
\author{Joan Birman}
\address{Department of Mathematics, Columbia University, 2990 Broadway, New York City, NY 10027}
\email{jb@math.columbia.edu}

\author{Peter Brinkmann}
\address{Google Inc.}
\email{peter.brinkmann@gmail.com}

\author{Keiko Kawamuro}
\address{Department of Mathematics, University of Iowa, 14 MacLean Hall, Iowa City, IA 52242}
\email{kawamuro@iowa.uiowa.edu}


\date{\today}

\begin{abstract} 
We investigate the structure of the characteristic polynomial $\det (xI-T)$ of a transition matrix $T$ that is associated to a train track representative of a pseudo-Anosov map $[F]$ acting on a surface. As a result we obtain three new polynomial invariants of $[F]$, one of them being the product of the other two, and all three being divisors of $\det (xI-T)$.  The degrees of the new polynomials are invariants of $[F]$ and we give simple formulas for computing them by a counting argument from an invariant train-track. 
We give examples of genus $2$ pseudo-Anosov maps having the same dilatation, and use our invariants to distinguish them.   
\end{abstract}

\maketitle

\section{Introduction}\label{sec:introduction}

Let $S$ be an orientable 2-manifold of genus $g$, closed or with finitely many punctures,
where the genus and the number of punctures are chosen so that $S$ admits a
hyperbolic structure.  The {\em modular group} $\Mod{S}$ is the  group
$\pi_0(\Diff{S})$, where admissible homeomorphisms preserve orientation.  
If a mapping class $[F] \in \Mod{S}$ is {\em pseudo-Anosov} or {\it pA}, then there exists
a representative $F :S\to S$, a pair of invariant transverse measured
foliations $({\mathcal F}^u, \mu_u), ({\mathcal F}^s, \mu_s),$  and a real number $\lambda$, the {\em dilatation} of $[F]$, such that $F$ multiplies the transverse measure $\mu_u$ (resp. $\mu_s$) by $\lambda$ (resp.  $\frac1\lambda$).  The real number $\lambda(F)$ is an invariant of the conjugacy class of $[F]$ in $\Mod{S}$.

In this paper we introduce a new approach to the study of invariants of $[F]$, when $[F]$ is pA.  Our work was in part motivated by recent efforts (see \cite{McM} and a host of papers that were inspired by it)  to understand precisely which real numbers $\lambda$ occur in this setting.   It is known that if one fixes $S$, then as $F$ is varied its dilatation $\lambda(F)$ takes on a minimum value  $\lambda_{S,min}$, where by this we mean that any pA map on $S$ has dilatation $\geq \lambda_{S,min}$, also $\lambda_{min,S}$ is realized by some pA map $F$. The number $\lambda_{S,min}$ is of great interest.  Many attempts have been made to use the approach in \cite{McM} to find $\lambda_{S,min}$ for closed surfaces of arbitrary genus $g$, however it appeared to us (after attending a workshop in April 2010 addressed to the study of $\lambda_{S,min}$), that  a new approach was needed. The main result in this paper is, for an arbitrary pA map $F$,  the determination of 
two integer polynomials, both containing $\lambda(F)$ as their largest real root, and the proof that both are invariants of the given pA mapping class.  Both will be seen to have a known topological meaning.   The study of these two polynomials, rather than of $\lambda$ itself, is the new approach that we have in mind.

{\em Measured train tracks} are a partially combinatorial device that Thurston introduced to encode the essential properties of $({\mathcal F}^u, \mu_u), ({\mathcal F}^s, \mu_s)$.   A train track $\tau$ is a branched 1-manifold that is embedded in the surface $S$.  It is made up of vertices (called {\em switches}) and smooth edges (called {\em branches}), disjointly embedded
in $S$.  See \S1.3 of \cite{PH}.
Given a pA map $[F]$, there exists a train track $\tau\subset S$ that {\em fills} the surface, i.e., the complement of $\tau$ consists of (possibly punctured) discs, and $\tau$ is left invariant by $[F]$.
Moreover, $\tau$ is equipped with a transverse measure (resp. tangential measure) that is related to the transverse measure $\mu_u$ on $\mathcal F_u$ (resp. $\mu_s$ on $\mathcal F_s$).

In \cite{BH} Bestvina and Handel gave an algorithmic proof of Thurston's
classification theorem for mapping classes.  
Their proof shows that, if $[F]$ is a pA map of $S$, then one may construct, algorithmically,  a graph $G$, homotopic to $S$ when $S$ is punctured, and an induced map $f:G \to G,$ that we call a {\it train track map}.  For every $r \geq 1$ the restriction of $f^r$ to the interior of every edge is an immersion. 
It takes a vertex of $G$ to a vertex, and takes an edge to an edge-path which has no backtracking.   
Let $e_1, \dots, e_n$ be the unoriented edges of $G$. Knowing $G$ and $f:G\to G$, they construct  a somewhat special measured train track $\tau$, and we will always assume that our $\tau$ comes from their construction.
The {\em transition matrix} $T$ is an $n \times n$ matrix whose entry $T_{i,j}$ is the number of times the edge path $f(e_j)$ passes over $e_i$ in either direction, so that all entries of $T$ are non-negative integers.
If $[F]$ is pA, then $T$ is irreducible and it has a dominant real eigenvalue $\lambda$, the {\em Perron-Frobenius eigenvalue} \cite{gantmacher2}. 
The eigenvalue $\lambda$ is the dilatation of $[F]$.  
The left (resp.\ right) eigenvectors of $T$ determine  tangential (resp.\  transversal) measures on $\tau$, and eventually determine $\mu_s$ (resp. $\mu_u$).

In this paper we study the structure of the characteristic polynomial $\det(xI-T)$ of the transition matrix $T$.
Our work depends crucially on the Bestvina-Handel algorithm, however we look for the structure needed to find a measured train track in $T$, and not in the matrix
$T'=
\scriptsize
\begin{bmatrix}
N & A\\
0 & T
\end{bmatrix}
$
that appears in  \S 3.3 and \S 3.4 of \cite{BH}.  Bestvina and Handel use $T'$ to build the invariant foliations associated to $f$. As is well-known, all of the information needed for that construction is already present in $T$, and we shall build on that fact. With that in mind, let $V(G)$ be the vector space of real weights on the edges of the Bestvina-Handel graph $G$.  
Let $f_*:V(G) \to V(G)$ be the map induced from the train track map $f:G \to G$.  Let $\chi(f_*)=\det (xI-T)$.
It is well-known that  $\chi(f_*)$ depends on the choice of $f:G\to G$ within its conjugacy class $[F]$. 

The first new result in this paper is the discovery that, after dividing $\chi(f_\ast)$ by a polynomial that is determined by the way that a train-track map acts on certain vertices of $G$, one obtains a quotient polynomial which is a topological invariant of $[F]$.  This polynomial arises via an $f_*$-invariant direct sum decomposition of the $\mathbb R$-vector space of transverse measures on $\tau$.  It is the characteristic polynomial of the action of $f_*$ on one of the summands.  We call it the {\it homology polynomial} of $[F]$ for reasons that will become clear in a moment.   We will  construct examples of pA maps on a surface of genus 2 which have the same dilatation, but are distinguished by their homology polynomials.   

Like $\chi(f_\ast)$, our homology polynomial  is the characteristic polynomial of an integer matrix, although (unlike $T$) that matrix is not in general non-negative.  We now describe how we found it.  
We define and study an $f_*$-invariant subspace $W(G,f) \subset V(G)$.   The subspace $W(G,f)$ is chosen so that weights on edges determine a transverse measure on the train track $\tau$ associated to $G$ and $f:G\to G$.  We study $f_* |_{W(G,f)}$.  See page 427 of \cite{T}, where the mathematics that underlies $f_\ast |_{W(G,f)}$ is described by Thurston.  Our first contribution in this paper is to make the structure that Thurston described there concrete and computable, via an enhanced form of the Bestvina-Handel algorithm. 
This allows us to prove that the characteristic polynomial $\chi(f_\ast |_{W(G,f)})$ is an invariant of the mapping class $[F]$ in $\Mod{S}$.  This polynomial $\chi(f_\ast |_{W(G,f)})$ is our homology polynomial  and we denote it by $h(x)$. It contains the dilatation of $[F]$ as its largest real root, and so is divisible by the minimum polynomial of $\lambda$.   
Its degree depends upon a careful analysis of the action of $f_\ast$ on the vertices of $G$.

Investigating the action of $f_\ast$ on $W(G,f)$, we show that $W(G,f)$ supports a skew-symmetric form that is $f_\ast$-invariant.  The existence of the symplectic structure was known to Thurston and also was studied by Penner-Harer in \cite{PH}, however it is unclear to us whether it was known to earlier workers that it could have degeneracies.  
See Remark~\ref{rem:completeness}. 
We discovered via examples that degeneracies do occur.  
In $\S$\ref{sec:2nd polynomial invariant} 
we investigate the radical $Z$ of the skew-symmetric form, i.e. the totally degenerate subspace of the skew-symmetric form, and arrive at an $f_\ast$-invariant decomposition of $W(G,f)$ as $Z \oplus (W(G,f)/Z)$.  This decomposition leads to a product decomposition of the homology polynomial as a product of two additional new polynomials,  with both factors being invariants of $[F]$.   We call the first of these new polynomials,
$p(x)=\chi(f_\ast |_Z)$, the {\it puncture} polynomial because it is a cyclotomic polynomial that relates to the way in which the pA map $F$ permutes certain punctures in $S$.   As for $s(x)=\chi(f_\ast |_{W(G,f)/Z})$, our {\it symplectic} polynomial,  we know that it arises from the action of $f_\ast$ on the 
symplectic space $W(G,f)/Z$, but we do not fully understand it at this writing.  Sometimes the symplectic polynomial is irreducible, in which case it is the minimum polynomial of $\lambda$.  However we will give examples to show that it can be reducible, and even an example where it is 
symplectically reducible. Thus its relationship to the minimum polynomial of $\lambda$ is not completely clear at this writing.

Summarizing, we will prove:
\begin{theorem}\label{thm:summarize} 
Let $[F]$ be a pA mapping class in $\Mod{S}$, with Bestvina-Handel  train track map $f:G\to G$ and transition matrix $T$. 
\begin{enumerate}
\item
The characteristic polynomial $\chi(f_\ast)$ of $T$ has a divisor, the homology polynomial $h(x)$ 
which is an invariant of $[F]$.  It contains $\lambda$ as its largest real root, and is associated to an induced action of $F_*$ on $H_1(X, \mathbb R)$, where $X$ is the surface $S$ when $\tau$ is orientable and its orientation cover $\tilde S$ when $\tau$ is non-orientable.
\item
The homology polynomial $h(x)$ decomposes as a product $p(x) \cdot s(x)$ of two polynomials, each a topological invariant of $[F]$.  
\begin{enumerate}
\item
The first factor, the puncture polynomial $p(x)$, records the action of $f_*$ on the radical of a skew-symmetric form on $W(G,f)$.  It has topological meaning related to the way in which $F$ permutes certain punctures in the surface $S$. 
It is a palindromic or anti-palindromic polynomial, and all of its roots are on the unit circle. 
\item
The second factor, the symplectic polynomial $s(x)$, records the action of $f_\ast$ on the non-degenerate symplectic space $W(G,f)/Z$.  It contains $\lambda$ as its largest real root.  It is palindromic.  If irreducible, it is the minimum polynomial of $\lambda$, but it is not  always irreducible.
\end{enumerate}
\item
The homology polynomial $h(x)$, being a product of the puncture and symplectic polynomials, is palindromic or anti-palindromic.
\end{enumerate}
\end{theorem}

The proof of Theorem~\ref{thm:summarize} can be found in $\S$\ref{sec:1st polynomial invariant} and $\S$\ref{sec:2nd polynomial invariant}  below.  In $\S$\ref{sec:applications} we give several applications, and prove that our three invariants behave nicely when the defining map $F$ is replaced by a power $F^k$.  The paper ends, in $\S$\ref{sec:examples} with a set of examples which give concrete meaning to our ideas.  The first such example, Example~\ref{ex:filling-curves}, defines three distinct maps $F_1,F_2,F_3$ on a surface of genus 2,  chosen so that all three have the same dilatation.  Two of the three pairs are distinguished by any one of our three invariants.  The third map was chosen so that it probably is not conjugate to the other two, however our invariants could not prove that.


\section{Proof of Part (1) of Theorem~\ref{thm:summarize}} \label{sec:1st polynomial invariant}

We begin our work in $\S$\ref{subsec:preliminaries} by recalling some well-known facts from \cite{BH} that relate to the construction of the train track $\tau$ by adding infinitesimal edges to the graph $G$. After that, in $\S$\ref{subsec:W(G,f)}, we introduce the space $W(G, f)$ of transverse measures, which plays a fundamental role throughout this paper.  Rather easily, we will be able to prove our first decomposition and factorization theorem. Thus at the end of $\S$\ref{subsec:W(G,f)} we have our homology polynomial in hand, but we do not know its meaning, have not proved it is an invariant, and don't know how to compute it.  In $\S$\ref{subsec:basis for W(G,f)}  we prepare for the work ahead by constructing a basis for $W(G,f)$.
 We also learn how to find the matrix for the action of $f_\ast$ on the basis.  With that in hand, in $\S$\ref{subsec:homology poly}  we identify the vector space $W(G,f)$ with a homology space. We will be able to prove Corollary~\ref{cor:homology polynomial is invariant}, which asserts that the homology polynomial $h(x)$ is a topological invariant of the conjugacy class of our pA map $[F] $  in ${\rm Mod}(S)$.  
(Later in $\S$\ref{sec:examples}, we will use it to distinguish examples of pA maps acting on the same surface and having the same dilatation.)


\subsection{Preliminaries}\label{subsec:preliminaries}
It will be assumed that the reader is familiar with the basic ideas of the algorithm of Bestvina-Handel \cite{BH}.  The mapping class $[F]$ will always be pA. Further, assume that we are given the graph $G\subset S$,  homotopic to $S$, and a train track map  $f:G \to G$.  We note that if $S$ is closed, the action of $[F]$ always has periodic points with finite order, and the  removal of a periodic orbit will not affect our results, therefore without loss of generality we may assume that $S$ is finitely punctured.   

Following ideas in \cite{BH} we construct a train track $\tau$ from $f:G\to G$ by equipping the vertices of $G$ with additional structure: 
Let $e_1, e_2 \subset G$ be two (non-oriented) edges originating at the same vertex $v$.  
Edges $e_1$ and $e_2$ belong to the same {\em gate} at $v$ if for some $r>0$, the edge-paths $f^r(e_1)$ and $f^r(e_2)$ have a nontrivial common initial segment.  
If $e_1$ and $e_2$ belong to different gates at $v$ and there exists some exponent $r>0$ and an edge $e$ so that $f^r(e)$ contains 
$e_2e_1$ or $e_1e_2$
as a subpath, then we connect the gates associated to $e_1, e_2$ with an {\em infinitesimal edge.}  
In this way, a vertex $v \in G$ with $k$ gates becomes an infinitesimal $k$-gon in the train track $\tau$.
While this $k$-gon may be missing one-side, the infinitesimal edges must connect all the gates at each vertex, see \S3.3 of \cite{BH}.
In addition to the infinitesimal edges, $\tau$ also has {\em real edges} corresponding to the edges of $G$.
Hence, a {\em branch} of $\tau$, in the sense of Penner-Harer \cite{PH}, is either an infinitesimal edge or a real edge.

It is natural to single out the following properties of the vertices of $G$:

\begin{definition}[Vertex types]\label{def:vertex-type}  
See Figure~\ref{fig:vertex-types}.
A vertex of $G$ is {\em odd} (resp.\ {\em even}) if its corresponding infinitesimal complete polygon in $\tau$ has an odd (resp.\ {\em even}) number of sides, and it is {\em partial} if its infinitesimal edges form a polygon in $\tau$ with one side missing.   
Partial vertices include the special case where $v$ has only
two gates connected by one infinitesimal edge; we call such vertices {\em evanescent}.  
The symbol $w_i$ (resp. $x_i$) will denote the weight of $i$-th gate (resp. infinitesimal edge).
\end{definition}

\begin{figure}[htpb!]
\centering
\psfrag{w0}{$w_0$}
\psfrag{w1}{$w_1$}
\psfrag{w2}{$w_2$}
\psfrag{w3}{$w_3$}
\psfrag{w4}{$w_4$}
\psfrag{x0}{$x_0$}
\psfrag{x1}{$x_1$}
\psfrag{x2}{$x_2$}
\psfrag{x3}{$x_3$}
\psfrag{x4}{$x_4$}
\psfrag{A}{\small odd}
\psfrag{B}{\small even}
\psfrag{C}{\small partial}
\psfrag{D}{\small evanescent}
\subfloat{
\includegraphics[width=0.2\textwidth]{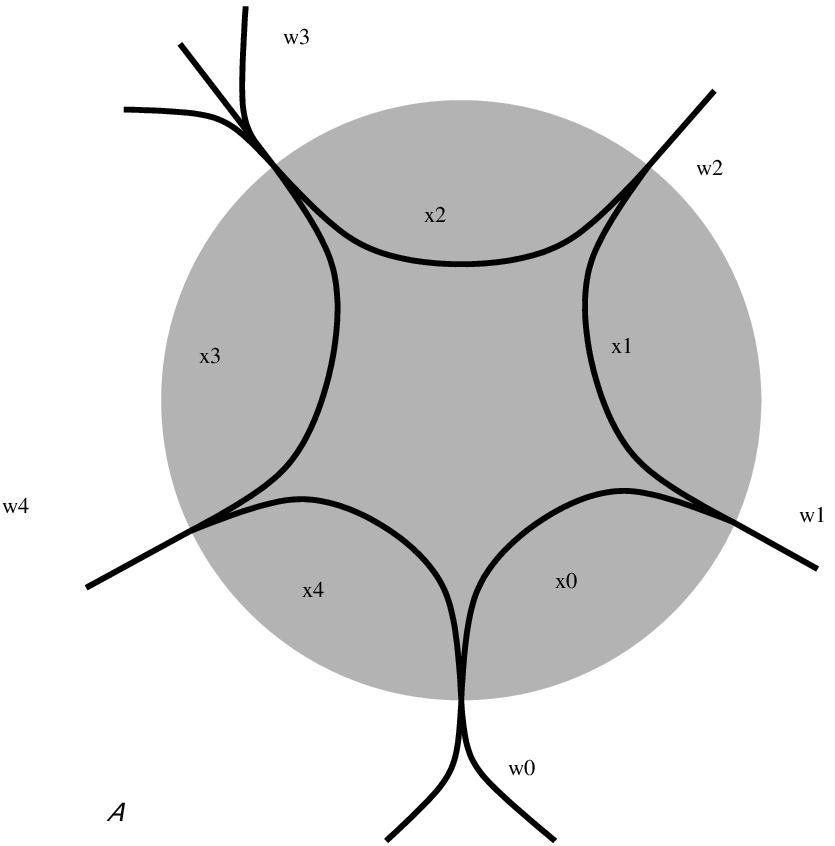}
\label{fig:oddvertex}}
\hfill
\subfloat{
\includegraphics[width=0.2\textwidth]{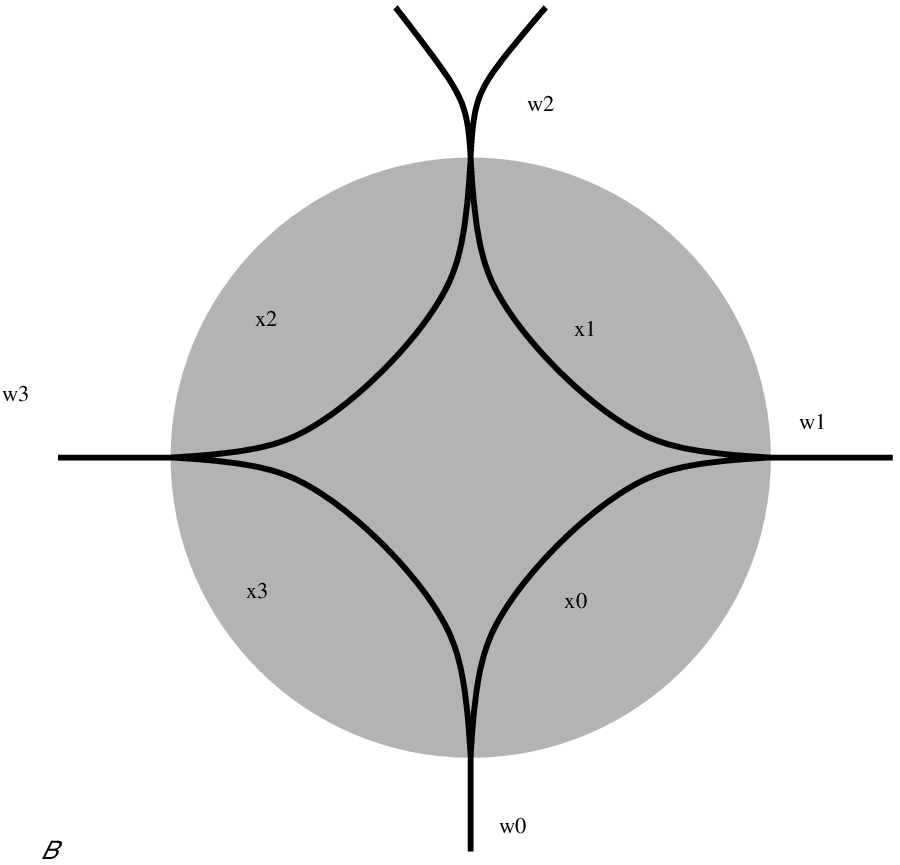}
\label{fig:evenvertex}}
\hfill
\subfloat{
\includegraphics[width=0.2\textwidth]{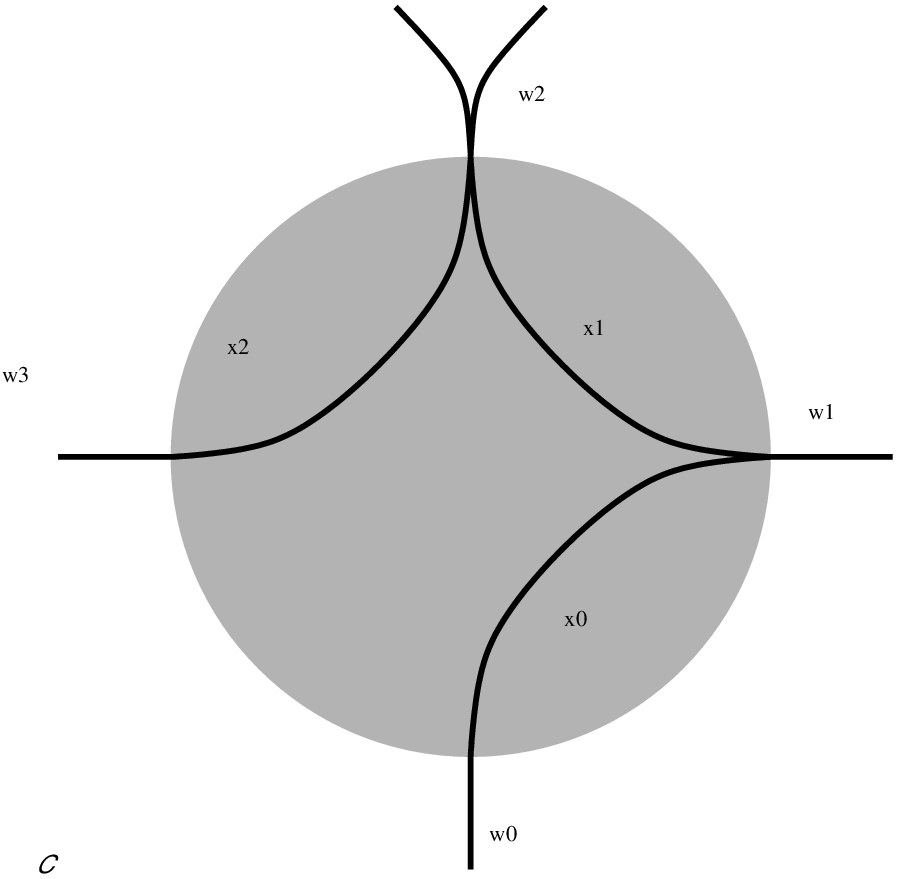}
\label{fig:partialvertex}}
\hfill
\subfloat{
\includegraphics[width=0.2\textwidth]{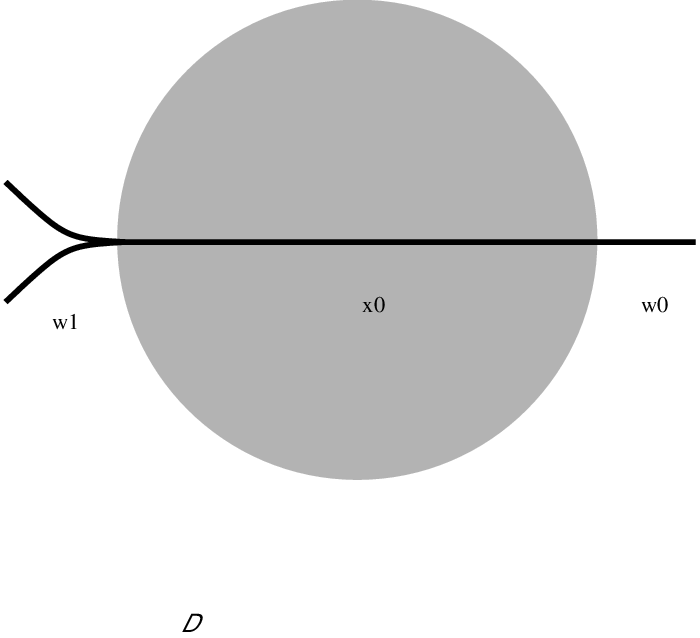}
\label{fig:evanescentvertex}}
\caption{Shaded disks enclose infinitesimal (partial) polygons in $\tau$ that correspond to the vertices of $G$.}\label{fig:vertex-types}
\end{figure}

\begin{remark}  
In $\S$\ref{sec:examples} the reader can find several examples illustrating the graph $G$ with infinitesimal polygons associated to particular pA mapping classes.  
In those illustrations the vertices of the graphs have been expanded to shaded discs which show the structure at the vertices.  
In the sketch of a train track that XTrain generates, the branches at each gate do not appear to be tangent to each other. 
This was done for ease in drawing the required figures.
The reader should keep in mind that all the branches at each gate are tangent to each other.
\end{remark}

We recall properties of non-evanescent vertices that are preserved under a train track map:

\begin{lemma}{\em(Proposition~3.3.3 of \cite{BH})}
\label{prop333}
For $k\geq 3$, let $O_k$ be the set of odd vertices with $k$ gates,
$E_k$ the set of even vertices with $k$ gates, and $P_k$ the set of
partial vertices with $k$ gates.  Then the restriction of $f:G\to G$ to each of these sets is a permutation of the set.

Moreover, for each $($non evanescent$)$ vertex $v$ with at least three gates, $f$ induces a bijection between the
gates at $v$ and the gates at $f(v)$ that preserves the cyclic order.
\end{lemma}

\begin{remark}
The number of evanescent vertices of a train track representative is {\em not} an invariant of the underlying mapping class.  Examples exist where a train track map has a representative with evanescent vertices, and another without. 
\end{remark}

\begin{definition}[Orientable and non-orientable train tracks]\label{def: tau orientable}
Choose an orientation on each branch of a train track $\tau$.   A train track is {\em orientable} if we can orient all the branches so that, at every switch, the angle between each incoming branch and each outgoing branch is $\pi$.  For example, see the train tracks $\tau_1$ and $\tau_2$ that are given in Figure~\ref{fig:filling-curves-tt} of $\S$\ref{sec:examples}.     After adding the infinitesimal edges, one sees that $\tau_1$ is orientable, but $\tau_2$ is not.  
\end{definition}
 
Here is an easy observation.

\begin{lemma}\label{vertextypelem} 
If $G$ has an odd vertex, then the corresponding train track $\tau$ is non-orientable.  This condition is sufficient, but not necessary.
\end{lemma}

\begin{proof}
If $v$ is an odd vertex, then there exists no consistent orientation for the corresponding infinitesimal polygon in $\tau$.  The example in Figure~\ref{fig:filling-curves-tt} shows that the condition is not necessary.
\end{proof}

\begin{remark}
We do not know any immediate visual criterion beyond the one in Lemma~\ref{vertextypelem}  for detecting non-orientability.  
The two train tracks $\tau_1,\tau_2$ in Figure~\ref{fig:filling-curves-tt} of this paper both have 2 vertices, one even and one partial, and $\tau_1$ is orientable whereas $\tau_2$ is non-orientable.   
If all the vertices are partial, then the train track may be either orientable (see Example~4.2 of \cite{pbexp}) or non-orientable (see  Example~\ref{ex:penner});  also, a non-orientable train track may have odd, even, and partial vertices at the same time (see Example~\ref{ex:evenodd}).
\end{remark}


\subsection{The space $W(G,f)$  and the first decomposition} \label{subsec:W(G,f)}

Given a graph $G$ of $n$ edges, one always has an $\R$-vector space $V(G) \simeq \R^n$ of weights on $G$.  
Our goal in this section is to define a subspace $W(G,f) \subset V(G)$ of `transverse measures on $G$'.  This space is the natural projection of the measured train track $\tau$ to a space of measures on $G$.    
It will play a fundamental role in our work.

In our setting, all train tracks are {\em bi-recurrent}, that is, recurrent and transversely recurrent, cf. p.20 of \cite{PH}. 
To define our space $W(G,f)$, we apply Penner-Harer's work described in \S3.2 of \cite{PH}, where bi-recurrence is assumed.

Let $V(\tau)\cong \R^{n+n'}$, where $n$ (resp. $n')$ is the number of the real (resp. infinitesimal) edges of train track $\tau$. 
Penner-Harer defined a subspace $W(\tau) \subset V(\tau)$ of assignments of (possibly negative) real numbers, one to each branch of $\tau$, which satisfy the {\em switch conditions.}   That is, if $\eta \in W(\tau)$ then at each switch of $\tau$, the sum of the weights on the incoming branches equals to the sum on the outgoing branches. 
For example, in Figure~\ref{fig:switch}-(A), $\eta(a) = \eta(b_1)+ \eta(b_2).$
\begin{figure}[htpb!]
\subfloat[]{
\psfrag{a}{$a$}
\psfrag{b}{$b_2$}
\psfrag{c}{$b_1$}
\includegraphics[width=0.2\textwidth]{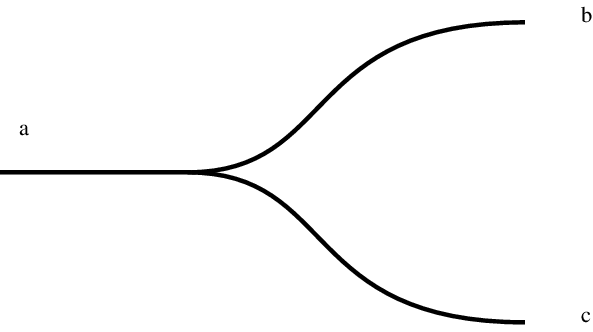}
}
\hspace{3cm}
\subfloat[]{
\psfrag{a}{$b_1$}
\psfrag{b}{$b_2$}
\psfrag{c}{$b_3$}
\psfrag{d}{$b_4$}
\psfrag{e}{$a$}
\includegraphics[height=2cm]{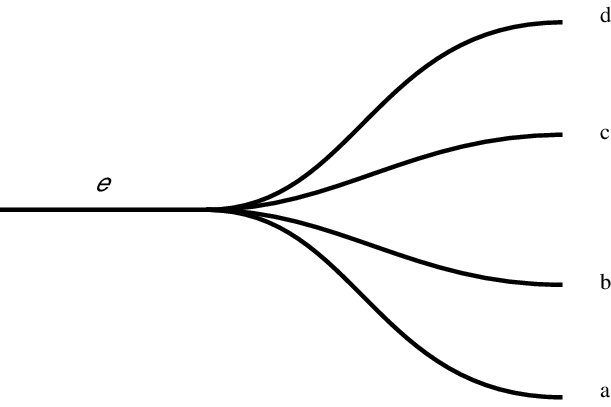}
\label{fig:valences}}
\caption{(A) A switch of valence $3$. (B) A switch of valence $5$.}\label{fig:switch}
\end{figure}

\begin{definition}
There is a natural surjection $\pi:\tau\to G$ which is defined by collapsing all the infinitesimal (partial) polygons to their associated vertices in $G$ and taking each real edge in $\tau$ to the corresponding edge in $G$.   
Let  $W(G, f) = \pi_\ast(W(\tau))$.   That is, $W(G, f)\subset V(G)$ is the subspace whose elements admit an extension to a (possibly negative) transverse measure on $\tau$.    
The name $W(G,f)$ has been chosen to reflect the fact that our subspace depends not only on $G$, but also on the action $f:G\to G$.  
\end{definition}

Here is a useful criterion for an element of $V(G)$ to be in $W(G,f)$:

\begin{lemma}\label{switchlem}
An element $\eta\in V(G)$ belongs to $W(G, f)$ if and only if for each non-odd vertex the alternating sum of the weights at the incident gates is zero.
\end{lemma}

\begin{proof}
Assume that $\eta\in V(G)$ belongs to $W(G, f)$.
Let $v \in G$ be a vertex with $k$ gates, i.e., $v$ corresponds to an infinitesimal $k$-gon, possibly partial, in the train track $\tau$. 
Let $w_0, \ldots, w_{k-1} \in \R$ be the weights of $\eta$  at the incident gates of $v$, and let $x_0, \ldots, x_{k-1} \in \R$ (or $x_0, \ldots, x_{k-2}$ if $v$ is a partial vertex) be the weights of the infinitesimal edges.  
See Figure~\ref{fig:vertex-types}.
The weights on the infinitesimal edges may turn out to be negative real numbers. 
We determine when an assignment of weights to the real edges admits an extension to the infinitesimal edges that satisfies the switch conditions.

If $v$ is odd or even with $k$ gates, then the switch condition imposes:
\[
\begin{array}{ccccc}
x_{k-1} & + & x_0 & = & w_0,\\
x_0 & + & x_1 & = & w_1,\\
x_1 & + & x_2 & = & w_2,\\
&& \vdots &&\\
x_{k-2} & + & x_{k-1} & = & w_{k-1}.\\
\end{array}
\]
If $k$ is odd, this system of equations has a unique solution, regardless of the weights $w_i$.  
If $k$ is even, the system is consistent if and only if $\sum_{i=0}^{k-1}(-1)^i w_i = 0$.

If $v$ is partial with $k$ gates, then the switch condition imposes:
\[
\begin{array}{ccccc}
& & x_0 & = & w_0,\\
x_0 & + & x_1 & = & w_1,\\
x_1 & + & x_2 & = & w_2,\\
&& \vdots &&\\
x_{k-3} & + & x_{k-2} & = & w_{k-2},\\
x_{k-2} & & & = & w_{k-1}.\\
\end{array}
\]
This system has a unique solution if and only if $\sum_{i=0}^{k-1} (-1)^i w_i = 0$.
\end{proof}

\begin{lemma}\label{eq:inclusion}
$W(G,f)$ is an invariant subspace of $V(G)$ under $f_\ast$, i.e.,
$f_\ast (W(G,f)) \subseteq W(G,f).$
\end{lemma}

\begin{proof}
Suppose $v$ is a non-odd vertex and mapped to a non-odd vertex $f(v)$. 
Let $\eta\in W(G,f)$. 
By Lemma~\ref{switchlem}, the alternating weight sum of $\eta$ at the incident gates of $v$ is $0$. 
Lemma~\ref{prop333} implies that all the weights of $\eta$ at the infinitesimal edges for $v$ is inherited to the weights of $f_\ast \eta$ at the infinitesimal edges for $f(v)$.

In addition, we account for an edge $e \subset G$ whose image $f(e)=e_0e_1\cdots e_k$ passes through the vertex $f(v)$. 
Assume that $\eta$ has weight $w=\eta(e)$ at the edge $e$.
If a sub-edge-path $e_ie_{i+1}$ passes through $f(v)$ then edges $e_i$ and $e_{i+1}$ belong to adjacent gates at $f(v)$ and the contribution of $e_i e_{i+1}$ to the alternating sum of weights of the gates at $f(v)$ is $w-w=0$.

Therefore, the alternating weight sum for $f_\ast \eta$ at the incident gates for $f(v)$ is $0$. 
By Lemma~\ref{switchlem}, $f_\ast \eta \in W(G,f)$.
\end{proof}

The dimension of the vector space $W(G,f)$ can be computed combinatorially by inspecting a train track associated to the pair $(G,f)$:
\begin{lemma}\label{lem:dimW(G,f)}
$(1)$ 
If $\tau$ is orientable, then 
\begin{eqnarray*}
\dim W(G,f)&=& \#( \mbox{edges of }G) - \#(\mbox{vertices of }G) +1, \\
W(G,f) &\cong& Z_1(G; \R) \cong H_1(G; \R) \cong H_1(S; \R).
\end{eqnarray*}
In particular, the switch conditions are precisely the cycle conditions.

$(2)$
If $\tau$ is non-orientable, then 
$$\dim W(G,f)= \#( \mbox{edges of }G) - \#(\mbox{non-odd vertices of }G).$$
\end{lemma}

\begin{proof}
Assume that $\tau$ is orientable. 
By Lemma~\ref{vertextypelem}, $G$ has no odd vertices. 
For $\eta \in W(G,f)$ and a non-odd vertex $v$, let 
$$w_i^v:= w_i^v(\eta)= \mbox{ the weight of } \eta \mbox{ at the } i^{\rm th} \mbox{ gate of the vertex }v.$$
By Lemma~\ref{switchlem},  $\sum_{i}(-1)^i w^v_i = 0$.
We number the gates at $v$ so that the orientation of the real edges at $2i$-th (resp. $(2i+1)$-th) gate is inward (resp. outward).  
If $G$ has $m$ vertices, $v_1, \cdots, v_m$, then we have a system of
$m$ equations:
\begin{eqnarray*}
w^{v_1}_0 + w^{v_1}_2 + w^{v_1}_4 + \cdots &=&
w^{v_1}_1 + w^{v_1}_3 + w^{v_1}_5 + \cdots, \\
w^{v_2}_0 + w^{v_2}_2 + w^{v_2}_4 + \cdots &=&
w^{v_2}_1 + w^{v_2}_3 + w^{v_2}_5 + \cdots, \\
& \vdots & \\
w^{v_m}_0 + w^{v_m}_2 + w^{v_m}_4 + \cdots &=&
w^{v_m}_1 + w^{v_m}_3 + w^{v_m}_5 + \cdots. 
\end{eqnarray*}
The sum of the left hand sides is equal to the sum of the right hand sides.
Since $\tau$ is oriented, the last equation follows from the other $m-1$ equations, i.e., the switch conditions are {\em not} independent.  
Therefore, 
\begin{eqnarray*} 
\dim W(G,f) &=& \dim V(G) - (m-1) \\
&=& \#(\mbox{edges of } G) - \#(\mbox{vertices of } G) +1. 
\end{eqnarray*}

With respect to the orientations of the edges of $G$, let $\partial : C_1(G; \R) \to C_0(G; \R)$ be the boundary map of the chain complex. 
There is a natural isomorphism  $V(G) \cong C_1(G; \R)$.
If $\gamma \in Z_1(G; \R)$ is a cycle, then the cycle condition $\partial \gamma=0$ is equivalent to the alternating sum condition $\sum_i (-1)^i w_i^v(\gamma) = 0$  at each vertex $v\in G$. 
By Lemma~\ref{switchlem} we obtain $W(G, f)\cong Z_1(G; \R)$. 
In fact, the Euler characteristic of $s$-punctured genus $g$ surface $S$ is;
$\chi(S)=2-2g-s = \#(\mbox{vertices of } G) - \#(\mbox{edges of } G).$
Thus $\dim W(G, f)=2g+s-1=\dim H_1(S; \R)$.

In the non-orientable case, the switch conditions are satisfied if and only if the alternating sum of the weights of gates around each even or partial vertex is zero (Lemma~\ref{switchlem}).  
Moreover, all these conditions are independent of each other (see Lemma 2.1.1 of \cite{PH}), so that the number of independent constraints is the number of non-odd vertices.  Since $\dim V(G)$ is the
number of edges, statement (2) follows.
\end{proof}

Now we know that $V(G)$ can be identified with the $1$-chains $C_1(G)$ in the orientable case, we can extend the boundary map $\partial : C_1(G) \to C_0(G)$ to the following map $\delta$ on $V(G)$:

\begin{definition}{\bf (The map $\delta$)}\label{def of delta}
Assume that the graph $G$ has $m$ non-odd vertices, $v_1, \cdots, v_m$. 
Define a linear map $\delta : V(G) \to \R^m,$ $\eta \mapsto \delta(\eta)$ such that 
$$(k^{\rm th} \mbox{ entry of the vector } \delta(\eta))=
\sum_i (-1)^i w_i^{v_k} (\eta),$$
the alternating sum of the weights of $\eta$ at the gates incident to the vertex $v_k$. These weights satisfy the following conditions:
 
\begin{itemize} 
\item 
If $\tau$ is oriented, then we determine the sign of each gate to be compatible with the orientation of the real edges of $\tau$. 
The alternating sum is defined without ambiguity. 
For an example, in Figure~\ref{fig:filling-curves-tt}, left sketch,  a  plus (resp. minus) sign may be assigned at each gate, according as the real edges are oriented  toward (resp. away from) the gate. 

\item 
If $\tau$ is non-orientable, we assign alternating signs to the incident gates for each non-odd vertex.  Since $\tau$ is non-orientable, the assignments will be local and not global. 
Clearly, the alternating sum depends on the choice of the sign assignment.  
\end{itemize}
\end{definition}

\begin{lemma}\label{lem:mapdelta}
In both the orientable and non-orientable cases $W(G, f) \cong \ker\delta$.  
Moreover, if $m$ is the number of non-odd
vertices of $G$, then
$$
\dim (\im\delta) = 
\left\{
\begin{array}{ll}
m - 1,& \mbox{if }\tau \mbox{ is orientable,} \\
m, & \mbox{if }\tau \mbox{ is non-orientable.}
\end{array}
\right.
$$
\end{lemma}

\begin{proof}
Lemma~\ref{switchlem}, immediately implies that $W(G,f) \cong \ker\delta.$ The dimension count follows from Lemma~\ref{lem:dimW(G,f)}.
\end{proof}

Finally we prove:

\begin{theorem}[First Decomposition]\label{thm:1st decomposition}
Let $h(x)=\chi(f_\ast|_{W(G,f)})$, then 
\begin{eqnarray} 
V(G) & \cong &  W(G,f) \oplus \im \delta.  \label{direct sum} \\
\chi(f_\ast) & = & h(x) \chi(f_\ast|_{\im\delta}). \label{polynomial}
\end{eqnarray}
The degree of $h(x)$ (resp. $\chi(f_\ast|_{\im\delta})$) is the dimension of 
$W(G,f)$ (resp. $\im\delta$), as given in Lemma~\ref{lem:dimW(G,f)} (resp. Lemma~\ref{lem:mapdelta}).
\end{theorem}

\begin{proof}
From Lemma~\ref{lem:mapdelta}, identifying $\im \delta$ with the quotient $V(G)/W(G,f)$ we obtain (\ref{direct sum}).
By the same argument as in the proof of Lemma~\ref{eq:inclusion}, we obtain $f_\ast (\im\delta) \subset \im\delta$. 
This along with $f_\ast (W(G,f)) \subset W(G,f)$ yields (\ref{polynomial}).  
\end{proof}


\subsection{A basis for $W(G,f)$ and the matrix for the action of  $f_\ast$ on $W(G,f)$} \label{subsec:basis for W(G,f)}  

The goal of this section is to find a basis for $W(G,f) \subset V(G)$ and learn how to determine the action of $f_\ast$ on the basis.
With regard to the basis, it will be convenient to consider three cases separately: the cases when $\tau$ is orientable,  when $\tau$ is non-orientable and has odd vertices, and when $\tau$ is non-orientable but has no odd vertices.  That is accomplished in $\S$\ref{sssec:basis1}, \ref{sssec:basis2} and \ref{sssec:basis3}.   Having the basis in hand, in $\S$\ref{sssec:basis4} we learn how to compute the action of $f_\ast$ on the basis elements. 

\subsubsection{Basis for $W(G,f)$, orientable case}\label{sssec:basis1}

If the train track $\tau$ associated to $G$ is orientable, we choose an orientation for $\tau$, thereby inducing an orientation on the edges of $G$, and choose a maximal spanning tree $Y \subset G$.  Every vertex of $G$ will be in $Y$.  We consider all edges $e$ of $G$ that are not in $Y$, and construct a set of vectors $\{\eta_e \in V(G)\}$ and prove that the constructed set is a basis for $W(G,f) \subset V(G)$.

For each $e\in G\setminus Y$, find the unique shortest path in $Y$  joining the endpoints of $e$.  
The union of this path and $e$ forms an oriented loop $L_e$ in which the edge $e$ appears exactly once, the orientation being determined by that on $e$. If the orientation of edge $e' \subset G$ agrees (resp. disagrees) with the orientation of $e'\subset L_e$, then we assign a weight of $1$ (resp. $-1$) to $e'$.   In particular, $e\subset L_e$ has weight 1. 
The edges not in $L_e$ are assigned weight of $0$.  In this way we obtain a vector $\eta_e \in V(G)$ whose entries are the assigned weights. 
By construction, $\eta_e$ satisfies the criterion for a transverse measure, described in Lemma~\ref{switchlem}, therefore $\eta_e$ is an element of $W(G,f)$. 

We now show that $\{\eta_{e_1},\dots,\eta_{e_l}\}$ is a basis for $W(G,f)$ where $e_1, \cdots, e_l$ are the edges of $G\setminus Y$  
and $l = \#( \mbox{edges of }G) - \#(\mbox{vertices of }G) +1.$
Note that if $e_i,e_j\in (G\setminus Y)$ with $i \not= j$, then $e_j\notin  L_{e_i}$. Therefore $\eta_{e_i}(e_j) = 1$ if and only if $i=j$, because all edges not in $ L_{e_i}$ have weight $0$, i.e., the vectors $\eta_{e_1}, \cdots, \eta_{e_l}$ are linearly independent. Consulting Lemma~\ref{lem:dimW(G,f)} we see that we have the right number of linearly independent elements, so we have found a basis for $W(G,f)$. 

\begin{example}\label{ex-of-basis-orientable}
Go to $\S$\ref{sec:examples} below and
see Example~\ref{ex:filling-curves} and its accompanying Figure~\ref{fig:filling-curves-tt}-(1).  The train track for this example is orientable.  The space $V(G)$ has dimension 5. Order the edges of $G$ as $a,b,c,d,e$.   The edge $a$ and  the vertices $v_0,v_1$ form a maximal tree $Y \subset G$, with edges $b,c,d,e \notin Y$, so that $W(G,f)$ has dimension 4.  We have the loops $L_b = ab$; $L_c= \overline ac$; $L_d = \overline ad$;  $L_e = ae$, so that $W(G,f)$ has basis  
$\eta_b = (1,1,0,0,0)'$; $\eta_c = (-1,0,1,0,0)'$; $\eta_d = (-1,0,0,1,0)'$; $\eta_e=(1,0,0,0,1)'$, where `prime' means transpose.   
\end{example}

\subsubsection{Basis for $W(G,f)$, non-orientable case with odd vertices}\label{sssec:basis2}

If $G$ has an odd vertex $v_0$ then choose a maximal spanning tree $Y \subset G$.
Let $V$ be the set of vertices of $G$.
Define a height function $h: V \to \N\cup \{0\}$ by 
$$h(v)= (\mbox{the distance between $v$ and $v_0$ in }Y).$$
We obtain a forest $Y' \subset Y$ by removing from $Y$ all the edges each of which connects an odd vertex and the adjacent vertex of smaller height. 
See Figure~\ref{forest}. 
\begin{figure}[htpb!]
\begin{center}
\psfrag{v0}{$v_0$}
\psfrag{v1}{$v_1$}
\psfrag{v2}{$v_2$}
\psfrag{v3}{$v_3$}
\psfrag{v4}{$v_4$}
\includegraphics[width=.9\textwidth]{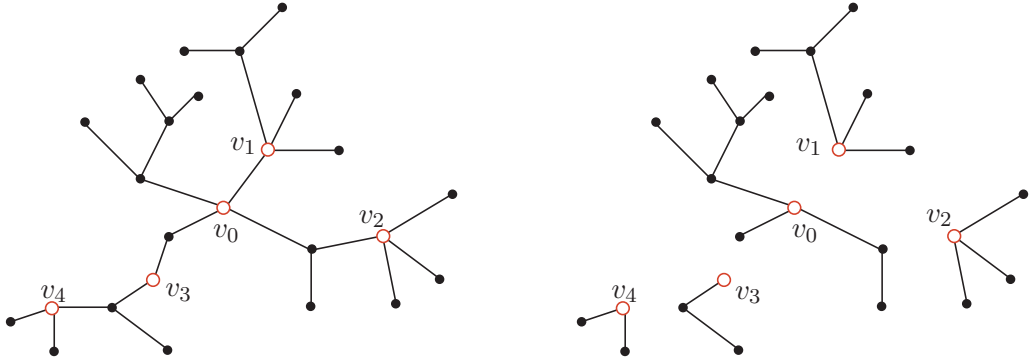}
\end{center}
\caption{A tree $Y$ (left) and a forest $Y'$ (right). 
Hollow dots $v_0, \cdots, v_4$, are odd vertices. Black dots are non-odd vertices.}
\label{forest}
\end{figure}
The forest $Y'$ contains all the vertices of $G$ with exactly one odd vertex in each connected component.  

Now, let $e$ be an edge that is not in $Y'.$  
We can find two (possibly empty) shortest paths in $Y'$ each of which connects an endpoint of $e$ to an odd vertex.
The union of $e$ and the two paths forms an arc, which we denote by $L_e$.
(If both of the endpoints of $e$ belong to the same tree component of $Y'$, then $L_e$ becomes a loop containing one odd vertex.)
Assign a weight of $1$ to $e$ and weight of $0$ to the edges that are not in $L_e$.
To the other edges in $L_e$, we assign weights of $\pm 1$ so that at each non-odd vertex the criterion of transverse measure (Lemma~\ref{switchlem}) is satisfied.
This defines an element $\eta_e$ of $W(G,f)$.

Let  $e_1, \cdots, e_l$ be the edges of $G \setminus Y'$.  By the construction, we have $\eta_{e_i}(e_j) = 1$ if and only if $i=j$,  so the vectors $\eta_{e_1}, \cdots, \eta_{e_l}$ are linearly independent.
Since $l = \#(\mbox{edges of }G) - \#(\mbox{non-odd vertices of }G),$
Lemma~\ref{lem:dimW(G,f)} tells us that $l = \dim W(G,f)$, hence $\{\eta_{e_1}, \cdots, \eta_{e_l}\}$ is a basis of $W(G,f)$. 

\begin{example}  
See Example~\ref{ex:evenodd}. 
The graph $G$ has two odd vertices $v_0$ and $v_4$.
We choose a maximal tree $Y$ whose edges are $a, c, d, j$.  This gives us a forest $Y'$ with two components. 
One consists of the single vertex $v_4$, and the other consists of vertices $v_0, v_1, v_2, v_3,$ and edges $a, c, j$.
The edge $h$ is not in $Y'$, and its endpoints are $v_1$ and $v_4.$  The associated arc $L_h = c h.$
Since $c$ and $h$ share the same gate at $v_1$, $\eta_h$ satisfies 
$\eta_h(c) = -1,$ $\eta_h(h) = 1$, and $0$ for rest of the edges.
The graph $G$ has 10 edges $a,b, \dots, j$,  with $a, c, j$ in $Y'$.  The vector space  $W(G,f)$ has dimension 7. The edge $h$ is not in $Y'$, and its endpoints are $v_1$ and $v_4.$   Then $L_h = ch.$  
Since $c$ and $h$ share the same gate at $v_1$, $\eta_h$ satisfies 
$\eta_h(c) = -1,$ $\eta_h(h) = 1$, and $0$ for rest of the edges.
The edge $b$ is also not in $Y'$, and its endpoints are $v_2$ and $v_3$,   so $L_b = cjbac$ is an arc whose endpoints coincide at $v_0$ and $\eta_b$ satisfies 
$\eta_b(a) = \eta_b(b) = \eta_b(j)=1$, and $0$ for rest of the edges.   The other five basis elements are constructed in a similar way.
\end{example}

\subsubsection{Basis for $W(G,f)$, non-orientable case with no odd vertices}\label{sssec:basis3}

In this case, we can find a simple loop $\mathcal L_0 \subset G$ that does not admit an orientation consistent with the train track $\tau$.  
If $\mathcal L_0$ misses any vertices of $G,$ then we define $\mathcal L_1$ by adding an edge with exactly one vertex in $\mathcal L_0.$ 
If $\mathcal L_1$ misses any vertices, we define $\mathcal L_2$ by adding an edge with exactly one vertex in $\mathcal L_1,$ etc. 
Ultimately we obtain a connected subgraph $\mathcal L$ that is homotopy equivalent to a circle and contains all vertices of $G$.  

If $e$ is an edge outside $\mathcal L,$ we can find paths in $\mathcal L$ from each endpoint of $e$ to the loop $\mathcal L_0,$  resulting in a path $L_\bullet$ that contains $e$ and with endpoints in $\mathcal L_0$.  
Now we can find paths $L_1, L_2$ in $\mathcal L_0$ so that $L_1 \cup L_2 = \mathcal L_0$ and the endpoints of $L_1, L_2$, called $v_a$ and $v_b$, agree with those of $L_\bullet$.
It is possible that $v_a$ and $v_b$ are the same vertex, in which case we set $L_2 = \emptyset.$  (This happens in Example ~\ref{ex-basis-non-ori2}, where our construction is applied to Example~\ref{ex:penner}.) 

Let $\eta^0 \in V(G)$ be a vector which assigns $1$ to edge $e$, $\pm 1$ to the other edges in $L_\bullet$, and $0$ to the edges not in $L_\bullet$, so that the alternating weight sum is $0$ at all the vertices but $v_a, v_b$.  

Next, let $\eta^1 \in V(G)$ (resp. $\eta^2 \in V(G)$) be a vector which assigns $\pm 1$ to the edges of $L_1$ (resp. $L_2$) and $0$ to the other edges not in $L_1$ (resp. $L_2$) so that $\eta^0 + \eta^1$ (resp. $\eta^0 + \eta^2$) has the alternating sum of weights equal to $0$ at all the vertices of $G$ but $v_a$. 
In particular, at vertex $v_b$, for $i=1,2$,
$(\mbox{the alternating sum of weights of } \eta^0 + \eta^i)=0,$
hence
$$(\mbox{the alternating sum of weights of } \eta^1 - \eta^2 \mbox{ at } v_b)=0.$$
While, at vertex $v_a$ 
$$(\mbox{the alternating sum of weights of } \eta^1 - \eta^2 \mbox{ at } v_a)=\pm 2.$$
For, if it were $0$ then the loop $\mathcal L_0$ can admit an orientation consistent with $\tau$, which is a contradiction. 
Therefore, at $v_a$,  
$$(\mbox{the signed weight of }\eta^1) = 
- (\mbox{the signed weight of }\eta^2).$$
This means, at $v_a$, we have
$(\mbox{the alternating sum of weights of } \eta^0 + \eta^1)=0$ 
if and only if
$(\mbox{the alternating sum of weights of } \eta^0 + \eta^2) \neq 0$.

If $(\mbox{the alternating sum of weights of } \eta^0 + \eta^1)=0$ then define $\eta_e := \eta^0 + \eta^1$ and $L_e := L_\bullet \cup L_1$.
Otherwise define $\eta_e := \eta^0 + \eta^2$ and $L_e := L_\bullet \cup L_2$.
By construction, $\eta_e$ satisfies the alternating sum condition of Lemma~\ref{switchlem}, hence $\eta_e \in W(G,f)$.

Note that the number of edges in $\mathcal L$ is equal to the number of vertices of $G$, so by Lemma~\ref{lem:dimW(G,f)} the number $l$ of edges outside $\mathcal L$ is equal to $\dim W(G,f)$. 
Suppose $e_1, \cdots, e_l$ are the edges of $G \setminus \mathcal L$, then $\eta_{e_i}(e_j) = 1$ if and only if $i=j$,  
i.e., vectors $\eta_{e_1}, \cdots, \eta_{e_l}$ are linearly independent.
This proves that $\{\eta_{e_1}, \cdots, \eta_{e_l}\}$ is a basis of $W(G,f)$.

\begin{example}\label{ex-basis-non-ori2}
See Example~\ref{ex:penner}. 
The partial vertex $v_0$ has ten gates.
We assign signs alternatively to the gates, which imposes orientations on the edges $b, c, e$.
However, the gates of each edge $a, d$ or $f$ have the same sign, hence $a, d, f$ do not admit consistent orientations.
We may choose the loop $\mathcal L=\mathcal L_0$ to be the union of $v_0$ and the edge $f$. 
For the edge $b$, the loop $L_b$ consists of a single edge $b$ and $v_0$. 
The element $\eta_b$ has 
$$\eta_b(b)=1, \mbox{ and } \ 
\eta_b(a)=\eta_b(c)=\eta_b(d)=\eta_b(e)=\eta_b(f)=0.$$
For the edge $a$, the loop $L_a = a \cup f$ and $\eta_a$ has
$$\eta_a(a)=\eta_a(f)=1, \mbox{ and } \
\eta_a(b)=\eta_a(c)=\eta_a(d)=\eta_a(e)=0.$$
The alternating sum of the weights of the ten gates is zero for both $\eta_a$ and $\eta_b$.
Lemma~\ref{switchlem} guarantees that $\eta_a, \eta_b \in W(G,f)$
\end{example}

\subsubsection{The matrix for the action of $f_*$ on $W(G,f)$} \label{sssec:basis4}
With respect to the basis of $W(G,f)$ described in $\S$\ref{sssec:basis1}, \ref{sssec:basis2} and \ref{sssec:basis3}, let $A$ denote the matrix representing the map $f_\ast |_{W(G,f)}$, 
With this $A$ we can compute the homology polynomial $h(x)$.

We compute $A$ explicitly as follows:   
Let $e_1, \cdots, e_n$ be the edges of $G$.
Let $\zeta_1, \cdots, \zeta_n$ be the standard basis of $V(G)\cong \R^n$, where $\zeta_i(e_j)= \delta_{i, j}$ the Kronecker delta.
Let $l=\dim W(G,f)$.
Suppose that $\{\eta_1, \cdots, \eta_l\}$ is a basis constructed as in $\S$\ref{sssec:basis1}, \ref{sssec:basis2} and \ref{sssec:basis3}.
Reordering the labels, if necessary, we may assume $\eta_j = \eta_{e_j}$ for $j = 1, \cdots, l$.  
Now let $Q$ be an $n \times l$ matrix whose entries $q_{i,j} \in \Z$ satisfy 
$\eta_j= \sum_{i=1}^n q_{i,j} \zeta_i$. 
Let $T$ be the transition matrix for the train track map $f:G\to G$, and let $P: \R^n \to \R^l$ be the projection onto the first $l$ coordinates. 
Then $A$ is an $l \times l$ integer matrix with
$A=PTQ.$
See Example 5.1 for the calculation of the matrix $A$. 
We note that Corollary~\ref{cor:isomorphism of f_ast},  in the next section,  implies that $A \in GL(l, \Z)$.

\begin{remark}
A related question is the computation of the `vertex polynomial' $f_\ast |_{\im\delta}$, even though that polynomial is not a topological invariant and so is only of passing interest.  We mention it because in special cases it may be easiest to compute the homology polynomial from the characteristic polynomial of $T$ by diving by the vertex polynomial.     
See examples~\ref{ex:k89} and \ref{ex:evenodd} below, where the computation of the vertex polynomial is carried out in two cases, using data that is supplied by XTrain.
\end{remark}

\subsection{The orientation cover and the homology polynomial} \label{subsec:homology poly} 

While we have the first decomposition theorem in hand,  and have learned how to compute the homology polynomial, that is the characteristic polynomial of the action of $f_\ast$ on $W(G,f)$, we have not proved that it is an invariant of $[F]$ and we do not understand its topological meaning.  All that will be remedied in this section.  Our work begins by recalling the definition of the orientation cover $\tilde S$ of $S$,  introduced by Thurston (p.427 of \cite{T}), see also p.184 of \cite{PH}. After that we will establish several of its properties. See Proposition~\ref{prop: orientation cover is orientable}.   In Theorem~\ref{thm:involution} and Corollary~\ref{cor:isomorphism of f_ast} we study the homology space $H_1(\tilde S; \R)$ and its relationship to our vector space $W(G,f)$ in the case when $\tau$ is non-orientable.  At the end of the section, in 
Corollary~\ref{cor:homology polynomial is invariant}, we establish the important and fundamental result that the homology polynomial is an invariant of the mapping class $[F]$ in Mod$(S)$.

\begin{definition} {\bf (Angle between two branches at a switch)} \label{def:angles at a switch} 
For a switch in $\tau$, fix a very small neighborhood that only contains the switch and the branches meeting at the switch. 
Within this neighborhood, we orient the branches in the direction outward from the switch.
This allows us to define the {\em angle} between two branches that meet at the switch.  
Since they always meet tangentially, this angle is either $0$ or $\pi$.  
If the angle is $0$, then we say that the branches form a {\em corner}.
For example, the angle between the branches $a$  and $b_1$ in Figure~\ref{fig:switch} (B) is $\pi$, whereas the angle between the branches $b_1$ and $b_2$ is $0$ and $b_1, b_2$ form a corner. 
\end{definition}

\begin{definition}{\bf (The orientation cover)}\label{def:orientation cover}
Let $F:S\to S$ be a pA homeomorphism with non-orientable train track $\tau$.  
Add a puncture to $S$ for each 2-cell of $S \setminus \tau$ corresponding to an odd or even vertex of $G$.  
The resulting surface $S'$ deformation retracts to
$\tau$, i.e., $\pi_1(S') = \pi_1(\tau)$.  
Each (not necessarily smooth) loop $\gamma \subset \tau$ consists of branches of $\tau$. 
We define a homomorphism 
$\theta:\pi_1(\tau)\to \mathbb Z/2 \mathbb Z$ 
which maps a loop in $\tau$ to $0$ if and only if it has an even number of corners (see Definition~\ref{def:angles at a switch}).  The {\em orientation cover} $p:\tilde S\to S$ associated to $\tau$ is obtained from the double cover $\tilde{S'}$ of $S$ corresponding to $\ker\theta$ by filling in the punctures in  $\tilde{S'}$ that do not belong to the original punctures of $S$.

At the same time, the non-orientable train track $\tau$ lifts to an orientable train track $\tilde\tau \subset \tilde S$. 
Collapsing the infinitesimal $($partial$)$ polygons in $\tilde\tau$ to vertices, we obtain a graph $\tilde G$, that is a double branched cover of $G$.
\end{definition}

Note that the branch points of $p: \tilde S\to S$ are precisely the odd
vertices of $G$.  
Intuitively, the effect of passing to the orientation cover
is a partial unrolling of loops in $\tau$ that do not admit a consistent
orientation.  
For example, Figure~\ref{fig:unroll} illustrates what happens near the branch point for a vertex of valence three.
\begin{figure}[htpb!]
\begin{center}
\psfrag{a}{$a$}
\psfrag{b}{$b$}
\psfrag{c}{$c$}
\psfrag{ta}{$\tilde a$}
\psfrag{tb}{$\tilde b$}
\psfrag{tc}{$\tilde c$}
\psfrag{tap}{$\tilde a'$}
\psfrag{tbp}{$\tilde b'$}
\psfrag{tcp}{$\tilde c'$}
\includegraphics[width=.65\textwidth]{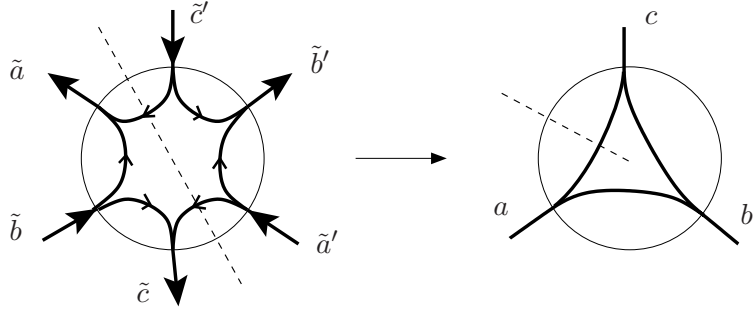}
\end{center}
\caption{Unrolling an odd vertex.}\label{fig:unroll}
\end{figure}

\begin{proposition}\label{prop: orientation cover is orientable}
Assume that $\tau$ is non-orientable.
\begin{enumerate}
\item The orientation cover $\tilde\tau$ is an orientable train track.
The natural involution $\iota: \tilde S\to \tilde S$, or the deck transformation, reverses the
orientation of $\tilde\tau$.
\item A puncture of $S$ corresponds to two punctures of $\tilde S$ if
and only if a loop around the puncture is homotopic to a loop in $\tau$
with an even number of corners.  
Otherwise, the puncture lifts to one puncture in $\tilde S$.
\end{enumerate}
\end{proposition}

\begin{proof}
(1) By Definition~\ref{def:orientation cover} we have $p_*(\pi_1(\tilde \tau))=\ker\theta$, hence every loop in $\tilde\tau$ has an even number of corners, and so $\tilde\tau$ can be consistently oriented.  If the
involution $\iota$ did not reverse the orientation of $\tilde\tau$, then
the orientation of $\tilde\tau$ would induce a consistent orientation of
$\tau$, but $\tau$ is not orientable.

(2) A loop in $\tau$ lifts to two loops in $\tilde\tau$ if and only if
it has an even number of corners.  If it has an odd number of corners, then its concatenation with itself has a unique lift.
\end{proof}

\begin{proposition}\label{or-cover-matrix}
A train track map $f:G\to G$ has two lifts
$\tilde f_{\rm op} : \tilde G\to \tilde G$, {\em orientation preserving}, and 
$\tilde f_{\rm or} : \tilde G\to \tilde G$, {\em orientation reversing}. 
They are related to each other by 
$\tilde f_{\rm or} = \iota \cdot \tilde f_{\rm op}$, 
where $\iota:\tilde G \to \tilde G$ is the deck translation.
Let $n$ be the number of edges in $G$.
Then there exist $n\times n$ non-negative matrices $A, B$ such that 
$A + B = T$, the transition matrix of $f$, and 
$\tilde f_{\rm op}$ and $\tilde f_{\rm or}$ are represented as:
$\tilde f_{\rm op} = 
\scriptsize
\begin{bmatrix}
A & B \\
B & A
\end{bmatrix}$
and 
$\tilde f_{\rm or} = 
\scriptsize
\begin{bmatrix}
B & A \\
A & B
\end{bmatrix}.
$
Their characteristic polynomials are
$\chi((\tilde f_{\rm op})_*) = \chi(f_*) \det(A-B)$ 
and
$\chi((\tilde f_{\rm or})_*) = \chi(f_*) \det(B-A).$
\end{proposition}

The above proposition confirms that pA maps, $F: S\to S$; $\tilde F_{\rm op}:\tilde S \to \tilde S$; and $\tilde F_{\rm or}: \tilde S \to \tilde S$, all have the same dilatation.

\begin{proof}
Even though the train track $\tau$ is non-orientable, we assign an orientation to each edge of $G$. 
Let $e_1, \cdots, e_n$ be the oriented edges of $G$. 
We denote the lifts of $e_k \subset G$ by $\tilde e_k, \tilde e_k' \subset \tilde G$.
Since $\tilde \tau$ is orientable, we choose an orientation, which induces orientations of $\tilde e_k, \tilde e_k'$. 
Denote the orientation cover by $p:\tilde G \to G$. 
Proposition~\ref{prop: orientation cover is orientable}-(1) implies that there are two choices:
$p(\tilde e_k)=e_k$ or $p(\tilde e_k)=\overline e_k$, where $\overline e_k$ is the edge $e_k \subset G$ with reversed orientation.
We choose to assume that $p(\tilde e_k)=e_k$, which implies that $p(\tilde e_k')=\overline e_k$.

We define an orientation preserving lift $\tilde f_{\rm op}: \tilde G \to \tilde G$ in the following way.
For an edge $e \subset G$, let $f(e)_{\rm head}$ (resp. $f(e)_{\rm tail}$) denote the first (resp. last) letter of the word $f(e)$. 
For each twin edges $\tilde e, \tilde e' \subset \tilde G$, we choose 
\begin{equation}\label{choice of head}
\tilde f_{\rm op} (\tilde e)_{\rm head}
:= 
\widetilde{f(e)_{\rm head}}, 
\qquad 
\tilde f_{\rm op} (\tilde e')_{\rm head}
:= 
\widetilde{f(e)_{\rm tail}}'.
\end{equation}
Next, we define the word $\tilde f_{\rm op} (\tilde e)$ to be the word $f(e)$ with each letter $e_i$ in $f(e)$ replaced by $\tilde e_i$ or $\tilde e_i'$ so that the resulting word corresponds to a connected edge-path in $\tilde G$. 
Due to the choice (\ref{choice of head}), the choice between $\tilde e_i$ and $\tilde e_i'$ is uniquely determined.
The word $\tilde f_{\rm op} (\tilde e')$ is given by the word $\tilde f_{\rm op} (\tilde e)$ {\em read from the right to left}, then replace $\tilde e_i$ by $\tilde e_i'$ and $\tilde e_i'$ by $\tilde e_i$.

We define an orientation reversing train track map by $\tilde f_{\rm or} := \iota \cdot \tilde f_{\rm op}$.

Let $\{ \zeta_1, \cdots, \zeta_n, \zeta_{1}', \cdots, \zeta_{n}' \}$ be the standard basis of $V(\tilde G) \simeq \R^{2n}$, where $\zeta_k, \zeta_k'$ correspond to $\tilde e_k, \tilde e_k' \subset \tilde G$ respectively.
From the constructions of $\tilde f_{\rm op}$ and $\tilde f_{\rm or}$, with respect to this basis, their transition matrices are of the form 
$\scriptsize
\begin{bmatrix}
A & B \\
B & A
\end{bmatrix}$
and 
$
\scriptsize
\begin{bmatrix}
B & A \\
A & B
\end{bmatrix}
$
respectively, for some non-negative $n\times n$ matrices $A$ and $B$ satisfying $A + B = T$. 
The formulae on characteristic polynomials follow from basic row and column reductions.
\end{proof}

In Example~\ref{ex:penner} there is a sketch of the orientation cover of 
a non-orientable train track which has no odd vertices. 
Also one can see explicit computations of $\tilde f_{\rm op}$ and $\tilde f_{\rm or}$ and matrices $A, B$.

We are finally in a position to understand the topological meaning of $W(G,f)$:

\begin{theorem}\label{thm:involution}
Assume that $\tau$ is non-orientable,   
Let $\iota: \tilde S\to \tilde S$ be the involution of the orientation cover.  Let $E^+$ and $E^-$ be the eigenspaces of $\iota_\ast: H_1(\tilde S; \RR)\to H_1(\tilde S; \RR)$ corresponding to the eigenvalues $1$ and $-1$, so that $H_1(\tilde S; \RR) \cong E^+ \oplus E^-.$  Then $E^+  \cong  H_1(S; \RR)$ and  $E^-  \cong  W(G,f)$.
\end{theorem}

\begin{proof}
Fix an orientation of $\tilde\tau$ once and for all.  This determines an orientation of $\tilde G$.  
Since $\iota$ is an involution, the only possible eigenvalues are $\pm 1$, and $\iota_\ast$ is diagonalizable.

For a homology class $\xi\in H_1(S; \RR)$, let $\tilde \xi\in H_1(\tilde S,
\RR)$ denote its lift to the orientation cover. Since $p\cdot\iota=p$, we have $\iota_\ast\tilde\xi = \tilde\xi$. Thus $H_1(S, \RR)\subseteq E^+.$

Each edge $e\subset G$ has two lifts $\tilde e$ and $\tilde e' \subset \tilde G$.  
Let $\zeta_e$ be the basis element of $V(G)$ corresponding to the (unoriented) edge $e$, and let $\zeta_{\tilde e}$ be the basis element
of $C_1(\tilde G; \R) \cong C_1(\tilde S; \R)$ corresponding to the (oriented) edge $\tilde e$.
Define a homomorphism: $\phi: V(G)\to C_1(\tilde S; \RR)$ by $\zeta_e \mapsto \zeta_{\tilde e}+\zeta_{\tilde e'}.$ 
Recall that a basis element $\eta_e$ of $W(G,f)$, introduced in 
$\S$~\ref{sssec:basis2} and $\S$~\ref{sssec:basis3}, has a corresponding arc or loop $L_e \subset G$ where $\eta_e$ assigns weight of $\pm 1$ satisfying the alternating sum condition. 
Different edges $e_1, e_2$ correspond to distinct $L_{e_1}$ and $L_{e_2}$. 
Moreover, when $L_e$ is an arc, its end points are odd vertices which are branch points of the orientation cover, so the lift of $L_e$ is a closed curve in $\tilde S$.  
Hence, the restriction of $\phi$ to $W(G, f)$ is injective. 
Assume $\eta_e = \sum_i \eta_e( e_i ) \zeta_{e_i}$.
By Proposition~\ref{prop: orientation cover is orientable}-(1), the involution takes $\iota: \tilde e \mapsto - \tilde e'$. We have:
$$\iota_\ast \phi(\eta_e) 
=
\iota_\ast (\sum_i \eta_e( e_i ) (\zeta_{\tilde{e_i}} + \zeta_{\tilde{e_i}'})) 
= \sum_i \eta_e( e_i ) (-\zeta_{\tilde{e_i}'} - \zeta_{\tilde{e_i}})
= - \phi(\eta_e),$$
i.e., $\phi(W(G,f))\subseteq E^-.$

Comparison of Euler characteristics along with Lemma~\ref{lem:dimW(G,f)} shows that 
$$
\dim H_1(\tilde S; \RR) =\dim H_1(S; \RR) + \dim W(G,f), 
$$
which implies $H_1(S; \RR) \cong E^+$ and $\phi(W(G,f)) \cong E^-$.
\end{proof}

In Lemma~\ref{eq:inclusion} we proved that $f_\ast (W(G,f)) \subseteq W(G,f).$ 
In fact, a stronger statement holds.

\begin{corollary}\label{cor:isomorphism of f_ast}
The restriction map $f_\ast |_{W(G,f)} : W(G,f) \to W(G,f)$ is an isomorphism.
\end{corollary}

\begin{proof}
Regardless of the orientability of $\tau$, the fact that $F: S \to S$ is a homeomorphism implies that the induced map $F_\ast : H_1(S; \R) \to H_1(S; \R)$ is an isomorphism.

Suppose that $\tau$ is orientable.
The isomorphism $W(G, f) \cong H_1(S; \R)$ in Lemma~\ref{lem:dimW(G,f)} allows us to identify $f_\ast |_{W(G,f)}$ with $F_\ast$, which is an isomorphism.

Suppose that $\tau$ is non-orientable.
Let $\{ \eta_e \}_{e \in E}$ be a basis of $W(G,f)$ constructed as in $\S$\ref{sssec:basis2} and $\S$~\ref{sssec:basis3}. 
Since the map $\phi: W(G,f) \to E^-$ in the proof of Theorem~\ref{thm:involution} is an isomorphism, the set $\{ \phi( \eta_e ) \}_{e\in E}$ is a basis of $E^-$.
Let $\tilde F: \tilde S \to \tilde S$ be a lift of $F: S \to S$.
It induces an isomorphism 
$\tilde F_\ast : H_1(\tilde S; \R) \to H_1(\tilde S; \R)$ and a train track map $\tilde f : \tilde G \to \tilde G$. 
Since $\tilde S$ deformation retracts to $\tilde G$, we can identify $\tilde F_\ast$ with $\tilde f_\ast : H_1(\tilde G; \R) \to H_1(\tilde G; \R)$. 
Using the same notation as in the proof of Theorem~\ref{thm:involution}, we have
\begin{eqnarray*}
\iota_\ast \tilde f_\ast (\phi (\eta_e)) 
& = &
\iota_\ast \tilde f_\ast 
\left( 
\sum_i \eta_e( e_i ) (\zeta_{\tilde{e_i}} + \zeta_{\tilde{e_i}'}) 
\right) \\
& = &
\iota_\ast \sum_i \eta_e( e_i ) (
\zeta_{\tilde f_\ast (\tilde{e_i})} + 
\zeta_{\tilde f_\ast (\tilde{e_i}')}
) \\
& = &
\sum_i \eta_e( e_i ) (
-\zeta_{\tilde f_\ast (\tilde{e_i}')} 
-\zeta_{\tilde f_\ast (\tilde{e_i})}
) 
= -\tilde f_\ast (\phi (\eta_e)). 
\end{eqnarray*}
Hence
$\tilde F_\ast (E^-) = \tilde f_\ast (E^-) \subset E^-$. 
Since $H_1(\tilde S; \R)$ is finite dimensional and $\tilde F_\ast$ is an isomorphism, we obtain that $\tilde F_\ast|_{E^-}= f_\ast|_{W(G,f)}$ is an isomorphism.
\end{proof}

The topological invariance of $\chi(f_\ast|_{W(G,f)})$ was stated as a conjecture in an earlier draft. Reading that draft, Jeffrey Carlson pointed the authors to a connection they had missed, making our conjecture an immediate consequence of Theorem~\ref{thm:involution}.  We are grateful for his help.

\begin{corollary}\label{cor:homology polynomial is invariant}  
Let $h(x)$ be the characteristic polynomial for $f_\ast|_{W(G,f)}$. It is the homology polynomial. Then $h(x)$ is an invariant of the pA mapping class $[F]$.
\end{corollary}

\begin{proof}
Let $\tilde F: \tilde S \to \tilde S$ be a lift of the pA map 
$F:S\to S$.  
By Theorem~\ref{thm:involution} we have
$h(x)=\chi(f_\ast|_{W(G,f)}) = \chi(\tilde F_\ast|_{E^-})$.  
Since the eigenspace $E^-$ is an invariant of $[F]$, so is the polynomial $h(x)$.
\end{proof}

This concludes the proof of Part (1) of Theorem~\ref{thm:summarize}.


\section{Proof of Parts (2) and (3) of Theorem~\ref{thm:summarize}}\label{sec:2nd polynomial invariant}
Having established the meaning of $W(G,f)$ and the invariance of the homology polynomial, our next goal is to understand whether it is irreducible, and if not to understand its factors.  At the same time, we will investigate its symmetries.  

With those goals in mind, we show that there is a well-defined and $f_*$-invariant skew-symmetric form on the space $W(G,f)$.  See  Proposition~\ref{prop:properties of skew-sym. form}.   We define  the subspace $Z\subset W(G,f)$ to be the space of degeneracies of this skew-symmetric form.  We are able to interpret the the action of $f_\ast$ on $Z$ geometrically, as being a permutation of certain punctures on $S$.  
   In Theorem~\ref{thm:2nd decomposition} we will prove that the space $W(G,f)$ has a decomposition into summands that are invariant under the action of $f_\ast$, and that as a consequence the homology polynomial decomposes as a product of two polynomials, $p(x)$ and $s(x)$. We call them the {\it puncture} and {\it symplectic} polynomials.  Like the homology polynomial, both are invariants of $[F]$ in Mod$(S)$.  We also establish their symmetries in Theorem~\ref{thm:2nd decomposition}, and understand the precise meaning of the puncture polynomial.  The symplectic polynomial contains $\lambda$ as its largest real root.  When irreducible, it coincides with the minimum polynomial of $\lambda$, but in general it is not irreducible.   At this writing we do not understand when it is or is not reducible.

\subsection{Lifting the basis elements for $W(G,f)$ to $W(\tau)$}\label{subsec:transitional & terminal}
Our work begins with a brief diversion, to establish a technical result that will be needed in the sections that follow.
We have shown how to construct basis elements $\eta_{e_1},\dots,\eta_{e_l}$ for $W(G,f)$.  We now build on this construction to give an explicit way to lift  each $\eta_e \in W(G,f)$ to an element $\eta_e' \in W(\tau)$ in such a way  that $\pi_\ast(\eta_e')=\eta_e$, where $\pi_* : W(\tau) \to W(G,f)$ is the natural surjection. 
Although there are infinitely many lifts of any given basis element $\eta_e$, our construction of a specific $\eta_e'$ will be useful later.   The issues to be faced in lifting $\eta_e$ to $\eta_{e'}$ are the assignment of weights to the infinitesimal edges.

\begin{definition}
Suppose that $v \in G$ is a vertex with $k$ gates numbered $0, \cdots, k-1,$ counterclockwise.
For $i= 0, \cdots, k-1$, we define a {\em transitional} element $\sigma_i \in V(\tau)$ that assigns $1$ to the $i$-th infinitesimal edge for the vertex $v$, and $0$ to the remaining branches of $\tau$. 
In other words, in Figure~\ref{fig:vertex-types}, $x_i = 1$ and $x_j = 0$ for $j\neq i$. 

Suppose that $v \in G$ is an odd vertex with $k$ gates.
For $i= 0, \cdots, k-1$, 
we define  a {\em terminal} element $\omega_i \in V(\tau)$ which assigns $\pm \frac12$ to the incident infinitesimal edges for $v$ 
so that the $i$-th gate has weight $w_i = 1$ and $j$-th $(j \neq i)$ gate has weight $w_j = 0$, cf. Figure~\ref{fig:weightedpath}; and assigns $0$ for rest of the branches of $\tau$.
\begin{figure}[htpb!]
\begin{center}
\psfrag{h}{$\frac12$}
\psfrag{mh}{$-\frac12$}
\psfrag{o}{$1$}
\psfrag{mo}{$-1$}
\psfrag{z}{$0$}
\includegraphics[width=.7\textwidth]{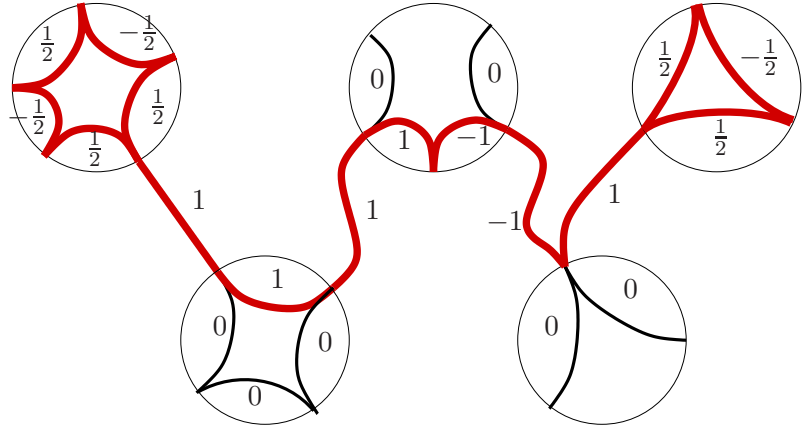}
\end{center}
\caption{Transitional and terminal elements.}
\label{fig:weightedpath}
\end{figure}
\end{definition}

For both the orientable and the non-orientable case, our basis element $\eta_e \in W(G,f)$ is a vector whose entries are $\pm 1$ or $0$. 
Recall that the edges whose weights are $\pm 1$ form a loop or an arc, denoted by $L_e$ in $\S$~\ref{sssec:basis1}, \ref{sssec:basis2} and \ref{sssec:basis3}.

At an even or a partial vertex of $L_e$, suppose $L_e$ goes through the $i$-th and $j$-th gates ($i \leq j$). 
To $\eta_e$ we add consecutive transitional elements $\sigma_i, \sigma_{i+1}, \cdots, \sigma_{j-1}$ with alternating signs so that the switch condition is satisfied (cf. the upper middle circle in Figure~\ref{fig:weightedpath}).
Repeat this procedure for all the non-odd vertices of $L_e$.
If $L_e$ is a loop, it yields an element $\eta_e'$ of $W(\tau)$.

When $L_e$ is an arc (i.e., $\tau$ is non-orientable with odd vertices), the two endpoints of $L_e$ are odd.
Suppose $L_e$ enters the $i$-th gate of an odd vertex.
After adding transitional elements as above, we further add terminal element $\omega_i$ or $-\omega_i$ so that the switch condition is satisfied at all the incident gates of the odd vertex. 
Proceed in this way for the other odd vertex as well, and we obtain an element $\eta_e'$ of $W(\tau)$.


\subsection{A skew-symmetric form on $W(G,f)$}\label{subsec:skew-symmetric form on W(G,f)}

In this section we define a skew-symmetric form $\langle \cdot , \cdot \rangle$ on $W(G,f)$. 
To get started, we slightly modify the skew-symmetric form on $W(\tau)$ introduced by Penner-Harer (p.182 of \cite{PH}):
For a branch $b \subset \tau$ and $\eta \in W(\tau)$, let $\eta(b)$ denote the weight that $\eta$ assigns to $b$.  
At a switch of valence $k$ ($k \geq 3$), label the branches $a, b_1, \cdots, b_{k-1}$ as in Figure~\ref{fig:switch}. 
The cyclic order of $a,  b_1, \cdots, b_{k-1}$ is determined by the orientation of the surface and the embedding of $\tau$ in $S$.
We define a skew-symmetric form:
$$
\langle \eta, \zeta \rangle_{W(\tau)} := 
\frac12 
\sum_{\substack{
\text{switches} \\
\text{in } \tau}
} \ 
\sum_{i < j}
\begin{vmatrix}
\eta(b_i) & \eta(b_j)\\
\zeta(b_i) & \zeta(b_j)
\end{vmatrix},
\quad 
\mbox{ for } \eta, \zeta \in W(\tau).
$$
Recall the surjective map $\pi : \tau \to G$ collapsing the infinitesimal (partial) polygons to vertices.

\begin{definition}
For $\eta, \zeta\in W(G,f)$ there exist $\eta', \zeta'\in W(\tau)$ so that $\pi_\ast(\eta') = \eta$ and $\pi_\ast(\zeta')=\zeta$.  We define a skew-symmetric form on $W(G,f)$ by:
$$\langle \eta, \zeta \rangle_{W(G,f)} := \langle \eta', \zeta' \rangle_{W(\tau)}$$ 
\end{definition}

\begin{proposition}  \label{prop:properties of skew-sym. form}
The skew-symmetric form  $\langle \cdot , \cdot \rangle_{W(G,f)}$ has the following properties:
\begin{enumerate}
\item It is well-defined.
\item When $\tau$ is orientable, $\langle \eta, \zeta \rangle_{W(G,f)} $ is the homology intersection number of $1$-cycles associated to $\eta$ and $\zeta$.  
\item  When $\tau$ is non-orientable, recall that $E^{\pm}$ are the eigenspaces of the deck transformation $\iota : \tilde S \to \tilde S$ for the orientation cover studied in Theorem~\ref{thm:involution}.
Since $p: \tilde S \to S$ is a {\em double} branched cover, we have the following results, to be compared with p.187 of \cite{PH}:
\begin{enumerate}
\item  The restriction of the intersection form on $H_1(\tilde S; \R)$ to $E^+$ is twice the intersection form on $H_1(S; \RR)$.
\item  The restriction of the intersection form on $H_1(\tilde S; \R)$ to $E^-$ is twice the skew-symmetric form $\langle \cdot, \cdot \rangle_{W(G,f)}$.
\end{enumerate}
\item For all $\eta, \zeta\in W(G,f)$, we
have $\langle f_\ast\eta, f_\ast\zeta \rangle_{W(G,f)} 
= \langle \eta, \zeta \rangle_{W(G,f)}.$
\end{enumerate}
\end{proposition}

\begin{proof}
(1) 
It suffices to show that for any $\eta' \in \ker \pi_\ast \subset W(\tau)$ and $\zeta' \in W(\tau)$, the product $\langle \eta' , \zeta' \rangle_{W(\tau)}=0$. 
Since $\pi_\ast(\eta') = \vec{0}$, $\eta'$ assigns $0$ to any real edge $b \subset \tau$, it follows that the weight of $\eta'$ at any gate is $0$. 
If a vertex $v\in G$ is odd or partial,  $\eta'$ assigns $0$ to any infinitesimal edges associated to the vertex $v$. 
Therefore, $v$ does not contribute to $\langle \eta' , \zeta' \rangle_{W(\tau)}$.
If a vertex $v\in G$ is even with $k$ gates, the weights $x_0, \cdots, x_{k-1}$ that $\eta'$ assigns to the infinitesimal edges for $v$ form an alternating sequence: $x_i = (-1)^i x_0$.
The contribution of the even vertex $v$ to $\langle \eta' , \zeta' \rangle_{W(\tau)}$ is:
$$
\frac{1}{2} \sum_{i=0}^{k-1}
\begin{vmatrix}
x_i & x_{i-1}\\
y_i & y_{i-1}\\
\end{vmatrix}
= \frac{x_0}2 \sum_{i=0}^{k-1}
\begin{vmatrix}
(-1)^i & (-1)^{i-1}\\
y_i & y_{i-1}\\
\end{vmatrix}
= \frac{x_0}2 \sum_{i=0}^{k-1}
(-1)^i(y_i+y_{i-1})
=0,
$$
where $y_i$ are the weights assigned by $\zeta'$ and indices are modulo $k$.

(2)
Assertion (2) is established in Lemma 3.2.2 of \cite{PH}.

(3)  
Assertion (3) follows directly from Theorem~\ref{thm:involution}.

(4)
If $\tau$ is orientable, then we identify $W(G,f)\cong H_1(S; \R)$. 
Since $F: S\to S$ is homeomorphism, the homology intersection number is preserved under $f_\ast : H_1(S ; \R) \to H_1(S ; \R)$ and the assertion follows. 

If $\tau$ is non-orientable, then, passing to the orientation cover, $\tilde F_\ast : H_1(\tilde S ; \R) \to H_1(\tilde S ; \R)$ preserves the homology intersection number. 
The assertion then follows from  (i) $\tilde F_\ast |_{E^-} = f_\ast|_{W(G,f)}$ and (ii) assertion (3) of this proposition. 
\end{proof}

Knowing that $\langle \cdot, \cdot \rangle_{W(G,f)}$ is well-defined, 
we can compute $\langle \eta_1, \eta_2 \rangle_{W(G,f)}= \langle \eta_1', \eta_2' \rangle_{W(\tau)}$ 
by using the basis elements $\eta_1, \eta_2 \in W(G,f)$ discussed in $\S$~\ref{sssec:basis1}, \ref{sssec:basis2} and \ref{sssec:basis3}, and their particular  extensions $\eta_1', \eta_2' \in W(\tau)$ introduced in $\S$~\ref{subsec:transitional & terminal}.
For this, it is convenient to study how transitional and terminal elements contribute to the skew-symmetric form. 
Straight forward calculation of determinants at incident gates yields the following:

\begin{proposition}\label{localcomputation}
Let $v$ be a vertex with $k$ incident gates, numbered $0,\ldots, k-1,$ conterclockwise.
We have 
$\langle \sigma_i, \sigma_j \rangle = \langle \sigma_0, \sigma_{j-i} \rangle$,
$\langle \omega_i, \sigma_j \rangle = \langle \omega_0, \sigma_{j-i} \rangle$
and
$\langle \omega_i, \omega_j \rangle = \langle \omega_0, \omega_{j-i} \rangle$  for $0\leq i\leq j\leq k-1$.
Moreover, 
$$
\langle \sigma_0, \sigma_i \rangle = 
\begin{cases}
-\frac12 & \text{if $i=1$}\\
\frac12 & \text{if $i=k-1$}\\
0 & \text{otherwise}
\end{cases}
\qquad
\langle \omega_0, \sigma_i \rangle = 
\begin{cases}
-\frac12 & \text{if $i=0$}\\
\frac12 & \text{if $i=k-1$}\\
0 & \text{otherwise}
\end{cases}
$$
$$
\text{ and }  \quad
\langle \omega_0, \omega_i \rangle = 
\begin{cases}
\frac{(-1)^i}2 & \text{if $i\neq 0$}\\
0 & \text{if $i=0$.}
\end{cases}
$$
\end{proposition}


\subsection{Degeneracies of the skew-symmetric form and the second decomposition} 
\label{subsec:radical}

In this section, we investigate the totally degenerate subspace of $W(G,f)$,
$$
Z := \{ \eta \in W(G,f) \ | \ 
\langle \zeta, \eta \rangle = 0 \ 
\mbox{ for all } 
\zeta \in W(G,f) \},
$$
the {\it radical} of the skew-symmetric form.   It will lead us, almost immediately, to the second decomposition theorem and another new invariant of pA maps.  We begin by showing how $Z$ has already appeared in our work, in a natural way.

\begin{proposition}\label{prop:dim Z}
Let $s$ be the number of punctures of $S$.
\begin{enumerate}
\item If $\tau$ is orientable,  $\dim Z=s-1$.
\item If $\tau$ is non-orientable, then
\begin{eqnarray*}
\dim Z & = &
\# ( \mbox{punctures of $S$ that correspond to two punctures in } \tilde S) \\
& = & 
\#(\mbox{punctures of $S$ represented by loops in $\tau$} \\
&& \quad \mbox{with even numbers of corners}).
\end{eqnarray*} 
\end{enumerate}
\end{proposition}

\begin{proof}
In the orientable case, recall Lemma~\ref{lem:dimW(G,f)} which states $W(G,f)  \cong H_1(S; \R)$, and Lemma 3.2.2 of \cite{PH} which shows that our skew-symmetric form agrees with the homology intersection form.
Hence the space $Z$ is generated by the homology classes of $s$ loops around the punctures.    
Because their sum is null-homologous, they are linearly dependent and $\dim Z = s-1$. 

In the non-orientable case, let $\tilde S$ be the orientation cover of $S$ (Definition~\ref{def:orientation cover}). Recall the eigen spaces $E^{\pm}$ for the deck transformation $\iota: \tilde S \to \tilde S$, cf. Theorem~\ref{thm:involution}.
Let $s$ (resp.\ $r$) be the number of punctures of
$S$ that lift to two (resp.\ single) punctures in $\tilde S$, and
let $\alpha_1, \beta_1, \ldots, \alpha_s, \beta_s, \gamma_1, \ldots,
\gamma_r$ be the homology classes of loops around the punctures of
$\tilde S$, chosen so that $\iota_\ast\alpha_i=\beta_i$ for all $i=1,
\ldots, s$ and $\iota_\ast\gamma_j = \gamma_j$ for all $j=1, \ldots, r$
and oriented so that their sum is zero.
The radical of the homology intersection form on $H_1(\tilde S; \R)$ is spanned by 
\begin{eqnarray*}
&&\spn{\alpha_1, \beta_1, \ldots,
\alpha_s, \beta_s, \gamma_1, \ldots, \gamma_r} \\
& = & \spn{\alpha_1-\beta_1, \ldots, \alpha_s-\beta_s, \alpha_1+\beta_1,
\ldots, \alpha_s+\beta_s, \gamma_1, \ldots, \gamma_r}.
\end{eqnarray*}
Note that $\alpha_i-\beta_i\in E^-$ and $\alpha_i+\beta_i, \gamma_j \in E^+$ for all $i=1, \ldots, s$ and $j=1, \ldots, r$.  
By Theorem~\ref{thm:involution} and assertion (3) of Proposition~\ref{prop:properties of skew-sym. form}, we obtain that  $Z = \spn{\alpha_1-\beta_1, \cdots, \alpha_s-\beta_s}$. 
Clearly, $\alpha_1-\beta_1, \cdots, \alpha_s-\beta_s$ are linearly independent, i.e., $\dim Z = s$.
\end{proof}

\begin{corollary}\label{rem:onepuncture}
Assume that $S$ is once punctured and that $\tau$ is non-orientable. Then $\dim Z =1$ if and only if $\dim W(G,f)$ is odd.
\end{corollary}

\begin{proof}
The induced skew-symmetric form on $W(G, f)/Z$
is non-degenerate, and so the dimension of $W(G, f)/Z$ is even.  
Thus, $\dim Z$ is odd if and only if $\dim W(G, f)$ is odd. 
Since $S$ has exactly one puncture, Proposition~\ref{prop:dim Z} yields that $\dim Z\leq 1$, and the corollary follows. 
\end{proof}

\begin{remark}\label{rem:completeness}
Straightforward modifications of our arguments show that if $\tau$ is a
non-orientable train track (not necessarily induced by a train track
map), then the dimension of $\rad W(\tau)$ is the number of
complementary regions of $\tau$ with even numbers of corners.  In
particular, if $\tau$ is complete, then the complement of $\tau$
consists of triangles and monogons (Theorem 1.3.6 of \cite{PH}), and
so the skew-symmetric form on $W(\tau)$ is non-degenerate in this case.
Although not explicit stated, this is the case covered by Theorem~3.2.4 of \cite{PH}. 
\end{remark}

\begin{theorem}[Second Decomposition]\label{thm:2nd decomposition}
Let $p(x)$ $($resp. $s(x))$ be the characteristic polynomial of $f_\ast|_Z$ $($resp. $f_\ast|_{W(G,f)/Z})$.
The map $f_\ast$ preserves the direct sum decomposition 
$W(G,f) \cong Z \oplus (W(G,f)/Z)$ so that
$h(x) = p(x) s(x).$
Moreover, we have:
\begin{enumerate}
\item 
The polynomial $p(x)$ is an invariant of the pA mapping class $[F] \in {\rm Mod}(S)$.
The restriction $f_\ast|_Z$ encodes how $F$ permutes the punctures whose projections to $\tau$ have even numbers of corners. 
In particular, $f_\ast|_Z $ is a periodic map, so that all the roots of $p(x)$ are roots of unity and the polynomial $p(x)$ is palindromic or anti-palindromic.  
\item 
The polynomial $s(x)$ is an invariant of $[F]$.
The skew-symmetric form $\langle \cdot, \cdot \rangle_{W(G,f)}$ naturally induces a symplectic form on $W(G,f)/Z$.  
The map $f_\ast$ induces a symplectomorphism of $W(G,f)/Z$.  Hence  $s(x)$  is palindromic.
\item 
The homology polynomial $h(x)$ is either palindromic or anti-palindromic.
\end{enumerate} 
\end{theorem}

\begin{proof}
Suppose $\eta \in Z.$ 
By assertion (4) of Proposition~\ref{prop:properties of skew-sym. form}, we have $0 =
\langle \eta, \zeta \rangle = \langle f_\ast(\eta), f_\ast(\zeta) \rangle$ for
all $\zeta\in W(G,f)$.  
By Corollary~\ref{cor:isomorphism of f_ast}, $f_\ast|_{W(G,f)}$ is surjective, and so $f_\ast(\eta)\in Z$.  
Thus $f_\ast$ preserves the decomposition $(W(G,f)/Z)\oplus Z$.

(1) 
The restriction $f_\ast|_Z$  is periodic because of Proposition~\ref{prop:dim Z} and the fact that $[F]$ permutes the punctures of $S$. 
Hence all the roots of $p(x)$ are roots of unity.  
Moreover, if $\mu$ is a root of $p(x)$, then $\frac1\mu =
\bar{\mu}$ is also a root of $p(x)$ because $p(x)\in\RR[x]$.  This implies
that $p(x)$ is palindromic or anti-palindromic.

(2)
By the definition of $Z$, the skew-symmetric form induces a non-degenerate form on $W(G,f)/Z$.  
This together with Proposition~\ref{subsec:skew-symmetric form on W(G,f)}-(4) implies that the polynomial $s(x)$ is palindromic.  It is an invariant of $[F]$ because it is the quotient of two polynomials, both of which have been proved to be invariants.  

(3) 
The homology polynomial $h(x)$ is either palindromic or anti-palindromic because it is a product of two polynomials, one of which is palindromic and the other of which is either palindromic or anti-palindromic.
\end{proof}

This concludes the proof of parts (2) and (3) of Theorem~\ref{thm:summarize}.   Since the proof of part (1) was completed in 
$\S$\ref{sec:1st polynomial invariant}, it follows that Theorem~\ref{thm:summarize} has been proved.     


\section{Applications}\label{sec:applications} In this section we give several applications of Theorem~\ref{thm:summarize}.   Corollary~\ref{cor:combinatorial invariants of [F]} summarizes the numerical class invariants of $[F]$ that, as a consequence of Theorem~\ref{thm:summarize}, can be computed from the train track $\tau$ by simple counting arguments.  
Corollary~\ref{cor on s(x)} is an application to fibered hyperbolic knots in 3-manifolds. 
Corollary~\ref{cor:powers} shows that our three polynomials behave very nicely under the passage $[F] \to[F^k]$.

\begin{corollary} \label{cor:combinatorial invariants of [F]}
Under the same notation as in Theorem~\ref{thm:summarize}, 
let $n, v, v_o$ be the number of edges, vertices, odd vertices respectively  in the
graph $G$.  
Let $s$ $($resp. $r)$ be the number of punctures of $S$ which are represented by loops in $\tau$ with even $($resp. odd$)$ numbers of corners.
Let $g$ $($resp. $\tilde{g})$ be the genus of $S$ $($resp. its orientation
cover $\tilde{S})$. 
\begin{enumerate}
\item The orientability $($or non-orientability$)$ of $\tau$ is a class invariant of $[F]$.
\item The degree of the homology polynomial  is a class invariant.  It is:
$$\deg  h(x) =
\begin{cases}
n - v + 1 & \text{if $\tau$ is orientable,}\\
n - v + v_o & \text{if $\tau$ is not orientable.}
\end{cases}
$$
\item The degree of the puncture polynomial is a class invariant. It is:
$$\deg p(x)=
\begin{cases}
s-1 & \text{if $\tau$ is orientable,}\\
s & \text{if $\tau$ is not orientable.}
\end{cases}
$$
\item  The degree of the symplectic polynomial is a class invariant.  It is:
$$\deg  s(x) =
\begin{cases}
2g & \text{if $\tau$ is orientable,}\\
2(\tilde g - g)& \text{if $\tau$ is not orientable.}
\end{cases}
$$  
\end{enumerate}
\end{corollary}

\begin{remark}
Assertion (4)  implies that the dilatation of $[F]$ is the largest real root of a polynomial of degree $2d$, where $2d \leq 2g$ $($resp. $2(\tilde g -g))$. However, this bound is not sharp because, as will be seen in  Example~\ref{ex:k89},  the symplectic polynomial is not necessarily irreducible.
\end{remark}

\begin{proof} 
Assertion (1) is clear.  See Lemma~\ref{lem:dimW(G,f)} for (2). 
In the orientable case each puncture is represented by a loop in $\tau$ with  an even number of corners.  This, together with Proposition~\ref{prop:dim Z}, implies (3).

To prove (4), let $v_e, v_p$  be the number of even vertices, partial vertices in $G$, respectively.  
If $\tau$ is orientable, the assertion is clear.    
If $\tau$ is non-orientable, the Euler characteristics of $S$ and $\tilde S$ are
\begin{eqnarray*}
\chi(S) &=& v_o + v_e + v_p - n = 2-2g-(r+s), \\
\chi(\tilde S) &=& v_o + 2v_e + 2v_p -2n=2-2 \tilde g - (r+2s). 
\end{eqnarray*}
From Lemma~\ref{lem:dimW(G,f)} and Proposition~\ref{prop:dim Z}, we have:
$$\dim (W(G, f) / Z)= (n -v_e - v_p) -s = 2 (\tilde g - g).$$
Note that the sum of the degrees of the symplectic and puncture polynomials is the degree of the homology polynomial.
\end{proof}

\begin{corollary}\label{cor on s(x)}
$(1) $
The symplectic polynomial $s(x)$ is an invariant of fibered hyperbolic links in $3$-manifolds. 

$(2)$ 
Assume that $M$ is a homology $3$-sphere and $K \subset M$ is a fibered hyperbolic knot whose monodromy admits an orientable train track. 
Let $\Delta_K(x)$ denote the Alexander polynomial of $K$. Then
$$
s(x) = \left\{
\begin{array}{ll}
\Delta_K(x) & \mbox{ if } f \mbox{ is orientation preserving, } \\
\Delta_K(-x) & \mbox{ if } f \mbox{ is orientation reversing. } 
\end{array}
\right. 
$$
\end{corollary}

\begin{proof}
(1)  
Let $L\subset M$ be a link.
Thurston proved that a 3-manifold $M \setminus L$ is fibered over $S^1$ with a pA monodromy $[F]\in {\rm Mod}(S)$ if and only if  $M\setminus L$ is hyperbolic. Combining his result with assertion (2) of Theorem~\ref{thm:2nd decomposition}, we obtain the first claim.

(2)
Let $F_\ast: H_1(S; \RR) \to H_1(S; \RR)$ be the induced map. 
By assertion (1) of Proposition~\ref{prop:dim Z}, the space $Z$ is trivial.
Lemma~\ref{lem:dimW(G,f)} tells us that $W(G,f)/Z \cong W(G,f) \cong H_1(S; \RR)$. Since ($\pm$)-gate is mapped to a ($\mp$)-gate if and only if $f : G \to G$ is orientation reversing, we have 
$$f_\ast|_{W(G,f)/Z} = \left\{
\begin{array}{ll}
F_\ast & \mbox{ if } f \mbox{ is orientation preserving, } \\
-F_\ast & \mbox{ if } f \mbox{ is orientation reversing. } 
\end{array}
\right. $$
The fact that $\Delta_K( \pm x)=\chi(\pm F_\ast)$ yields the statement. 
\end{proof}

\begin{remark}
Corollary~\ref{cor on s(x)}-(2) can be seen as a refinement of Rykken's  Theorem 3.3 in \cite{rykken}. 
\end{remark}

Our final application is to prove that our three polynomials,  $h(x), p(x)$ and $s(x)$ behave in a very nice way under the passage $[F] \to [F^n]$:

\begin{corollary}\label{cor:powers}
Let $n > 0.$
If $f: G \to G$ represents a pA mapping class $[F]$, then $f^n$ represents $[F^n]$. 
Suppose $s([F], x) = \prod_i (x-z_i)$ and $p([F], x) = \prod_j (x-w_j)$ where $z_i, w_j \in \mathbb{C}$, then 
\begin{eqnarray*}
s([F^n], x) &=& \prod_i (x-z_i^n), \\
p([F^n], x) &=& \prod_j (x-w_j^n),\\
h([F^n], x) &=& \prod_i (x-z_i^n)\prod_j (x-w_j^n).
\end{eqnarray*}
\end{corollary}

\begin{proof}
Note that the pA maps $[F]$ and $[F^n]$ act on the same surface and share the same graph $G$ and the associated train track $\tau$.
The direct sum decomposition in Theorem~\ref{thm:2nd decomposition} tells us that 
$f^n_\ast|_{W(G,f^n)/Z} = (f_\ast|_{W(G,f)/Z})^n$ and $f^n_\ast|_Z = (f_\ast|_Z)^n.$
Since $z_i$ and  $w_j$ are eigenvalues of $f_\ast|_{W(G,f)/Z}$ and $f_\ast|_Z$ respectively, the desired equations follow.  The product decomposition for the homology polynomial follows from the fact that it is a product of the other two polynomials.
\end{proof}


\section{Examples}\label{sec:examples}

All of our examples were analyzed with the software package XTrain
\cite{pbexp}, with some help from Octave \cite{octave}.  This package is an adaptation of the Bestvina-Handel algorithm to once-punctured surfaces. Our illustrations
show a train track $\tau$ (in the sense of \cite{BH}) embedded in
a once-punctured surface.  Regardless of whether $\tau$ admits an orientation,
we equip individual edges of the graph $G$ with a direction for the purpose of
specifying the map $f: G\to G$, although they coincide when $\tau$ is
orientable.   We remark that the limitations of the available software, at this time,  to once-punctured surfaces means that the puncture polynomial in our examples is always either 1 or $x-1.$ 

A surface is shown as a fundamental domain in the
Poincar\' e model for $\mathbb H^2$, with the identification pattern on the boundary given
by the labels on edges.  For example, a side crossed by an edge labeled
$a$ will be identified with the side crossed by the edge labeled
$\bar{a}$.  These labels also indicate the direction of an edge; $a$ is
the first half, $\bar{a}$ the second one.  The shaded regions in the pictures
contain all the infinitesimal edges associated with a vertex, as in
Figure~\ref{fig:vertex-types}.   To recover the graph $G$, collapse each shaded
region to a point.  In each example, we give the associated train track map on
edges of $G$.  That map determines the maps on the vertices.

\begin{example} \label{ex:filling-curves} 
We  illustrate Corollary~\ref{cor:homology polynomial is invariant}   with a  triplet of examples.  
Sketches (1), (2) and (3) of Figure~\ref{fig:filling-curves} show three copies of a once-punctured genus 2 surface $S= S_{2, 1}$, each containing two simple closed curves $u_i$ and $v_i$.   In all three cases  $u_i\cup v_i$  fills $S$, that is the complement of  the union of the two curves is a family of discs.  In all three cases the geometric intersection $i(u_i\cap v_i)  = 6.$   
\begin{figure}[htpb!] 
\centerline{\includegraphics[scale=.5] {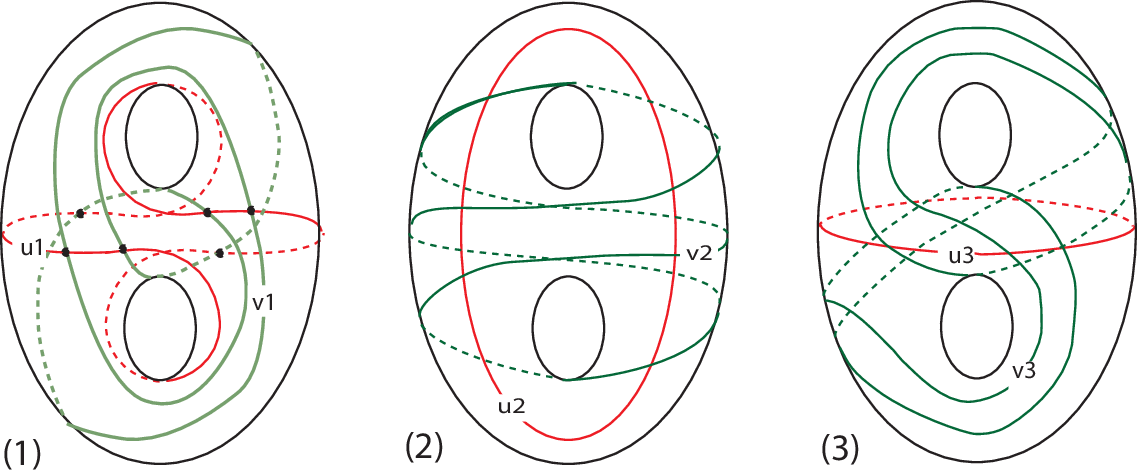}}
\caption {Curves on a surface of genus 2.}\label{fig:filling-curves}
\end{figure}
We use our curves to define three diffeomorphisms of $S$ by the formula  $F_i = T_{v_i}^{-1} T_{u_i}$ where $T_c$ denotes a Dehn twist about a simple closed curve $c$.   
By  a theorem of Thurston (Theorem 14.1 of \cite{FM}) there is a representation of the free subgroup of $\rm{Mod}(S)$ generated by $T_{u_i}$ and $T_{v_i}$ in ${\rm PSL}(2, \mathbb R)$ which sends the product
$F_i$  to  the matrix $\scriptsize \left(\begin{matrix} 1 & 0 \\ -6 & 1 \end{matrix}\right)^{-1}  \left(\begin{matrix} 1 & 6 \\ 0 & 1 \end{matrix}\right)  =
 \left(\begin{matrix} 1 & 6 \\ 6 & 37 \end{matrix}\right).$   By the Thurston's theorem, $F_i$ is pA and its dilation is the largest real root of the characteristic polynomial $x^2-38x+1$ of this matrix, that is  $37.9737\dots$ in all three cases.   
 We ask two questions: Are $F_1,F_2,F_3$ conjugate in ${\rm Mod}(S)$, and if they are not conjugate can our invariants distinguish them?   
 
See \cite{pbexp} for a choice of {\it standard curves} $a_0,d_0,c_0,d_1,c_1$ on $S$.   To obtain the needed input data for the computer software XTrain \cite{pbexp} we must express our maps as products of Dehn twists about these curves.  For simplicity, we denote the Dehn twist $T_{a_0}$ by the same symbol $a_0$. Using $A_i, C_i, D_i$ for the inverse Dehn twists of $a_i, c_i, d_i$, we find, after a small calculation, that:
\begin{eqnarray*}
F_1 
& = & 
c_0 d_0 d_1 A_0 C_1 c_1 d_1 c_0 d_0 A_0 D_0 C_0 D_1 C_1 c_1 a_0 D_1 D_0 C_0  \\
&&
c_1 b_0 (C_1 D_1)^6
c_1 d_1 c_0 d_0 a_0 D_0 C_0 D_1 C_1 (d_1 c_1)^6 
B_0 C_1 \\
F_2 
& = & 
c_1 (d_1 c_1)^6 a_0
c_1 d_1 c_0 d_0 A_0 D_0 C_0 D_1 C_1
A_0 (C_1 D_1)^6 C_1
c_1 d_1 c_0 d_0 a_0 D_0 C_0 D_1 C_1 \\
F_3 
& = & A_0D_0c_0d_0a_0B_1D_1c_0d_1b_1A_0D_0B_1D_1C_0d_1b_1d_0a_0B_1D_1 C_0d_1b_1A_0D_0C_0 d_0a_0(d_1c_1)^6
\end{eqnarray*}
Focussing on $F_1$ and $F_2$ first,  XTrain tells us that the associated transition matrices are:
\begin{equation*}
T_1 = 
\begin{bmatrix}
15&7&14&23&16  \\
10&6&10&16&11 \\
4&2&5&7&5  \\
2&1&2&3&1  \\
10&5&10&16&12 
\end{bmatrix}
\ {\rm and} \ \
T_2 =
\begin{bmatrix}
6 & 10 & 5 & 6 & 10\\
5 & 11 & 5 & 7 & 10\\
5 & 10 & 6 & 6 & 10\\
6 & 12 & 6 & 7 & 12\\
5 & 10 & 5 & 5 & 11
\end{bmatrix}
\end{equation*}
\noindent In both cases $\chi(f_\ast) = {\det}(xI-T_i) = x^5 -41x^4 + 118 x^3 -118 x^2 + 41 x -1,$
with largest real root $37.9737\dots$, as expected.

XTrain tells that the train tracks $\tau_1,\tau_2$ for our two examples are the ones that are  illustrated in Figure~\ref{fig:filling-curves-tt}. 
\begin{figure}[htpb!] 
\psfrag{a}{$a$} 
\psfrag{b}{$b$} 
\psfrag{c}{$c$} 
\psfrag{d}{$d$}
\psfrag{e}{$e$}
\psfrag{A}{$\overline a$} 
\psfrag{B}{$\overline b$} 
\psfrag{C}{$\overline c$} 
\psfrag{D}{$\overline d$}
\psfrag{E}{$\overline e$}
\psfrag{i}{(1) $\tau_1$, orientable}
\psfrag{ii}{(2) $\tau_2$, non-orientable}
\psfrag{0}{$v_0$}
\psfrag{1}{$v_1$}
\centerline{\includegraphics[scale=.6]{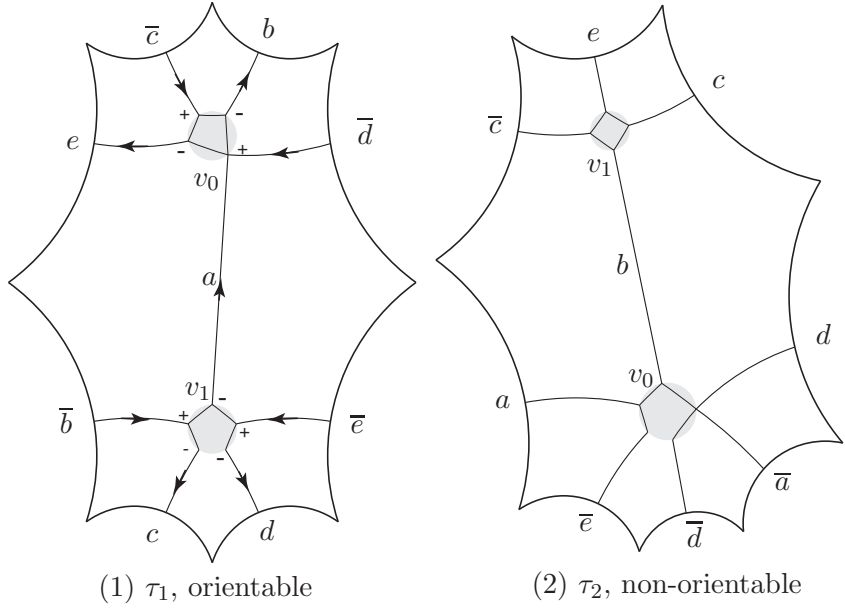}}
\caption {Train tracks for the maps $F_1,F_2$ of Example~\ref{ex:filling-curves}.} \label{fig:filling-curves-tt}
\end{figure}
With the train tracks $\tau_1,\tau_2$ in hand we can see, immediately, that $F_1$ and $F_2$ are inequivalent, because $\tau_1$ is orientable and $\tau_2$ is not.  Also, by Lemma~\ref{lem:dimW(G,f)} the dimension of $W(G_i, f_i)$, which is the degree of the homology polynomial, is $4$ (resp. $3$) when $i=1$ (resp. $2$).

We compute the homology and symplectic polynomials of $F_1$ and $F_2$ explicitly.

For $F_1$, a basis of $W(G_1, f_1)$, $\{\eta_b,\eta_c,\eta_d,\eta_e\}$,  were computed in Example~\ref{ex-of-basis-orientable}.  
Set $e_1 = b, e_2=c, e_3=d, e_4=e, e_5=a$ and follow the instructions in $\S$\ref{sssec:basis4} for finding the matrix $A_1$ representing $(f_1)_\ast |_{W(G_1, f_1)}$.  We obtain:   
\begin{equation*}
A_1= 
\scriptsize{
\begin{bmatrix}
1&0&0&0&0  \\
0&1&0&0&0 \\
0&0&1&0&0  \\
0&0&0&1&0   
\end{bmatrix}
\begin{bmatrix}
6&10&16&11&10 \\
2&5&7&5&4 \\
1&2&3&1&2  \\
5&10&16&12&10 \\
7&14&23&16&15  
\end{bmatrix}
\begin{bmatrix}
1&0&0&0 \\
0&1&0&0  \\
0&0&1&0\\
0&0&0&1  \\
1&-1&-1&1   
\end{bmatrix}
}
=
\scriptsize{
\begin{bmatrix}
16& 0& 6& 21\\
6 & 1 & 3 & 9 \\
3 & 0 & 1 & 3\\
15& 0& 6& 22
\end{bmatrix}
}
\end{equation*}
Its characteristic polynomial is $1 - 40 x + 78 x^2 - 40 x^3 + x^4 = (-1+x)^2 (1-38x +x^2),$
which is the homology polynomial.
Corollary~\ref{cor:combinatorial invariants of [F]}-(4) tells that the symplectic polynomial has degree $4$, hence it coincides with the homology polynomial.

For $F_2$, we apply $\S$\ref{sssec:basis3} and set the non-orientable loop $\mathcal L_0 = a \cup v_0$ and subgraph $\mathcal L = \mathcal L_1= a \cup b \cup v_0 \cup v_1$. Edges $c, d, e$ are not in $\mathcal L$.
We reorder the edges and call $e_1=c, e_2=d, e_3=e, e_4=a, e_5=b$. With this order, basis vectors of $W(G_2, f_2)$ are
$\eta_a = (1, 0, 0, 1, 2)'$; 
$\eta_c = (0, 1, 0, 0, 0)'$;
$\eta_e = (0, 0, 1, 0, -1)'$;
where ``prime'' means the transpose. Following the instructions in $\S$\ref{sssec:basis4}, we obtain the matrix $A_2$ representing $(f_2)_*|_{W(G_2, f_2)}$.
\begin{equation*}
A_2 =
\scriptsize{
\begin{bmatrix}
1&0&0&0&0  \\
0&1&0&0&0 \\
0&0&1&0&0   
\end{bmatrix}
\begin{bmatrix}
6 &  6 &  10 &  5 &  10\\
6 &  7 &  12 &  6 &  12\\
5 &  5 &  11 &  5 &  10\\
5 &  6 &  10 &  6 &  10\\
5 &  7 &  10 &  5 &  11
\end{bmatrix}
\begin{bmatrix}
1&0&0  \\
0&1&0 \\
0&0&1  \\
1&0&0\\
2&0&-1   
\end{bmatrix}
}
=
\scriptsize{
\begin{bmatrix}
31 &  6 &  0\\
36 &  7 &  0\\
30 &  5 &  1
\end{bmatrix}
}
\end{equation*}
The homology polynomial is
$\det(xI-A_2)=-1 + 39 x - 39 x^2 + x^3 =(-1 + x) (1 - 38 x + x^2)$, which means $\dim W(G_2, f_2)=3$. By Corollary~\ref{rem:onepuncture}, the symplectic polynomial has degree $2$, so it is $1 - 38 x + x^2$.

We turn to $F_3$.  From XTrain, we learn that the homology polynomials for $F_2$ and $F_3$ are the same.  However, observe that the curves $u_1,v_1,u_2,v_2,v_3$ are all non-separating on $S$, but $u_3$ is separating.  From this it follows that there cannot be an element $F'\in {\rm Mod}(S)$ that  maps $u_3$ to either $u_i$ or $v_i, \ i=1,2$.  Thus $[F_3]$ is very likely not conjugate to either $[F_1]$ or $[F_2]$ in ${\rm Mod}(S)$.  
\end{example}

\begin{example}[Figure~\ref{fig:k89}]\label{ex:k89} 
This example shows that the symplectic polynomial need not be irreducible over the rationals. The monodromy of the hyperbolic knot $8_9$ \cite{knotsandlinks} is represented by the following train track map:
\begin{figure}[htpb!]  
\begin{center}
\input{k8_9.psfrag}
\includegraphics[width=0.5\textwidth]{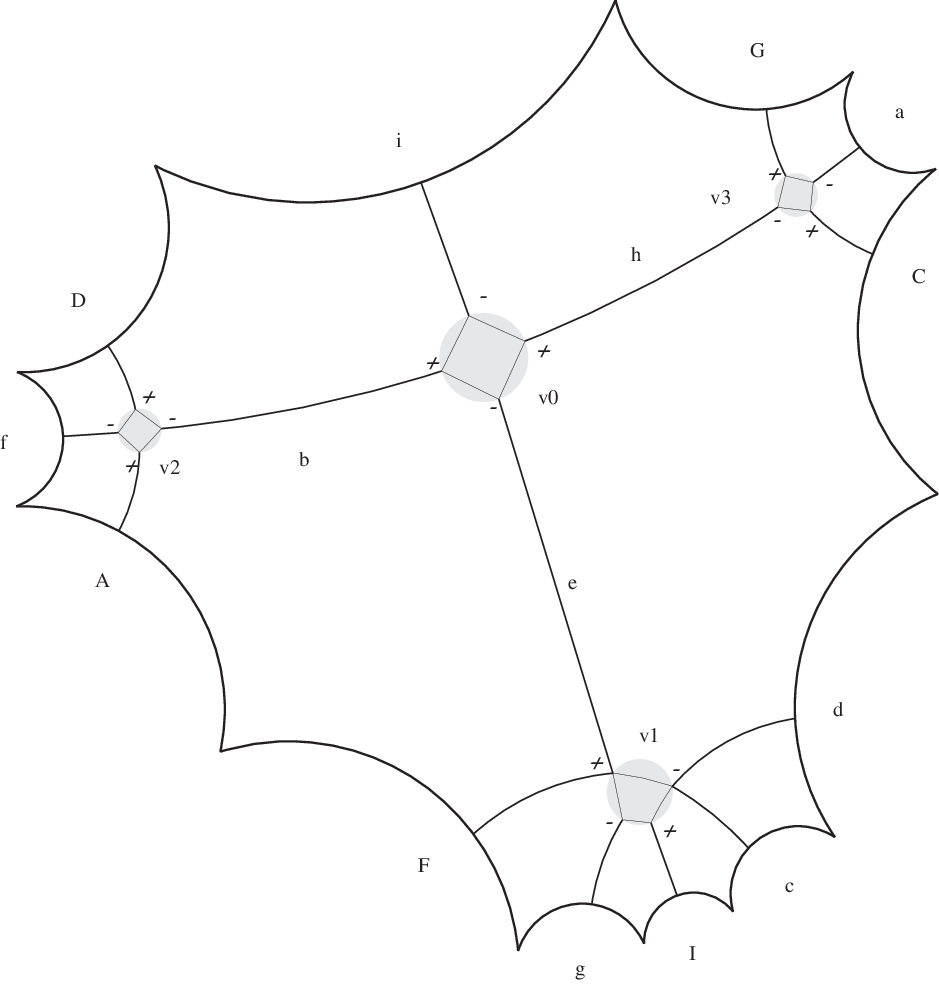}
\end{center}
\caption{Example~\ref{ex:k89}: The knot $8_9$. Orientable train track.}
\label{fig:k89}
\end{figure}
$$
\begin{array}{lll}
a: (v_{3}, v_{2}) \mapsto e &
b: (v_{2}, v_{0}) \mapsto g &
c: (v_{1}, v_{3}) \mapsto b\\
d: (v_{1}, v_{2}) \mapsto bi &
e: (v_{0}, v_{1}) \mapsto hed &
f: (v_{2}, v_{1}) \mapsto d\\
g: (v_{1}, v_{3}) \mapsto fgh&
h: (v_{3}, v_{0}) \mapsto ic &
i: (v_{0}, v_{1}) \mapsto a
\end{array}
$$
We have that $\chi(f_\ast) = x^9 -2x^8 +x^7 -4x^5 +4x^4 -x^2 +2x -1$.
Fix a basis of $\im\delta$, $\{ v_0 - v_1, v_2-v_0, v_3-v_0\}.$ With respect to this basis, $f|_{\rm im \delta}$ is represented as:  
$$\scriptsize
\begin{bmatrix}
0 & -1 & 0 \\
-1 & 0 & 0 \\
1 & -1 & -1 
\end{bmatrix}
$$
whose characteristic polynomial is $x^3+x^2-x-1$.  
Therefore $h(x)$ is: 
$$(x^9 -2x^8 +x^7 -4x^5 +4x^4 -x^2 +2x -1)/(x^3+x^2-x-1) = x^6-3x^5+5x^4-7x^3+5x^2-3x+1.$$
Proposition~\ref{prop:dim Z}-(1)
tells us that $p(x) = 1$, so that $s(x) = x^6-3x^5+5x^4-7x^3+5x^2-3x+1$.   
Since $f:G\to G$ is orientation preserving, by Corollary~\ref{cor on s(x)}-(2), we have $\Delta_{8_9}=x^6-3x^5+5x^4-7x^3+5x^2-3x+1$.
It further factors as $(x^3-2x^2+x-1)(x^3-x^2+2x-1)$.  
It is interesting that these factors are no longer palindromic, and one contains the dilatation $\lambda$ as a root and the other  $1/\lambda$. 
\end{example}

A similar analysis based on the hyperbolic knot $8_{10}$ shows that $s(x)$ need not even be symplectically irreducible.
It has a non-orientable train track and 
$$
s(x)= (x+1)^2
(x^{10}-3x^9+3x^8-4x^7+5x^6-5x^5+5x^4-4x^3+3x^2-3x+1). 
$$

The train tracks in the remaining examples are all non-orientable.

\begin{example}[Figure~\ref{fig:penner}]\label{ex:penner}
This example was suggested to us by Robert Penner.  
Figure~\ref{fig:pennersub1}
\begin{figure}[htpb!]  
\begin{center}
\psfrag{c1}{$c_1$}
\psfrag{c2}{$c_2$}
\psfrag{c3}{$c_3$}
\psfrag{c4}{$c_4$}
\psfrag{c5}{$c_5$}
\subfloat[Positive (inverse) Dehn twists about $c_1, c_2, c_3$, ($c_4, c_5$).]{\includegraphics[width=0.45\textwidth]{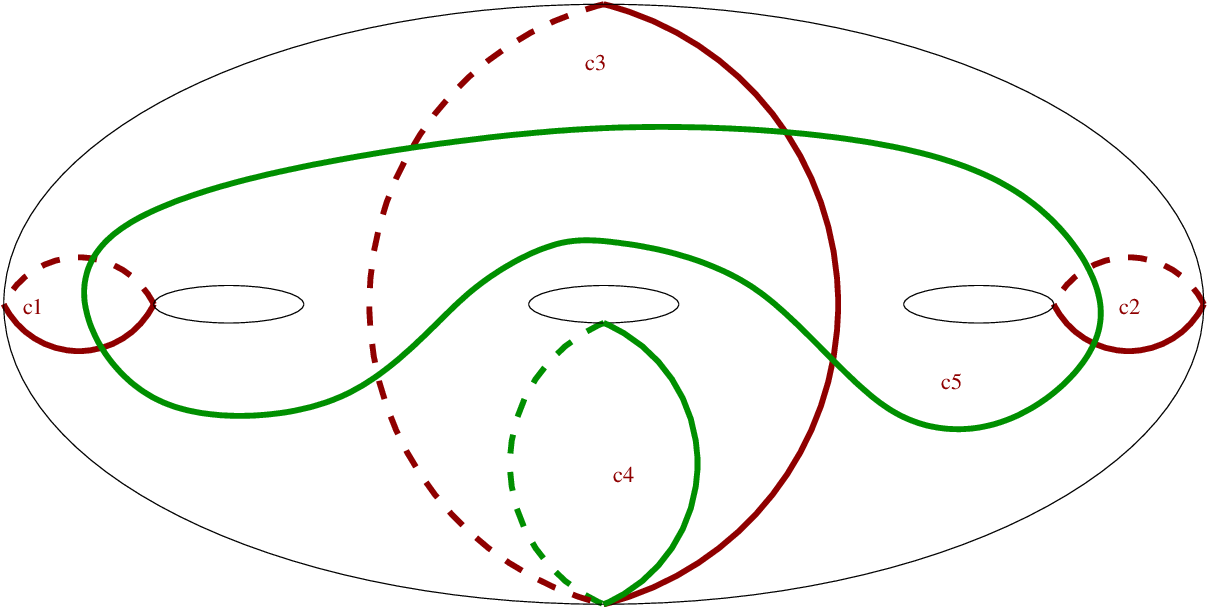}
\label{fig:pennersub1}
}
\input{penner.psfrag}
\subfloat[Train track $\tau$]{\includegraphics[width=0.45\textwidth]{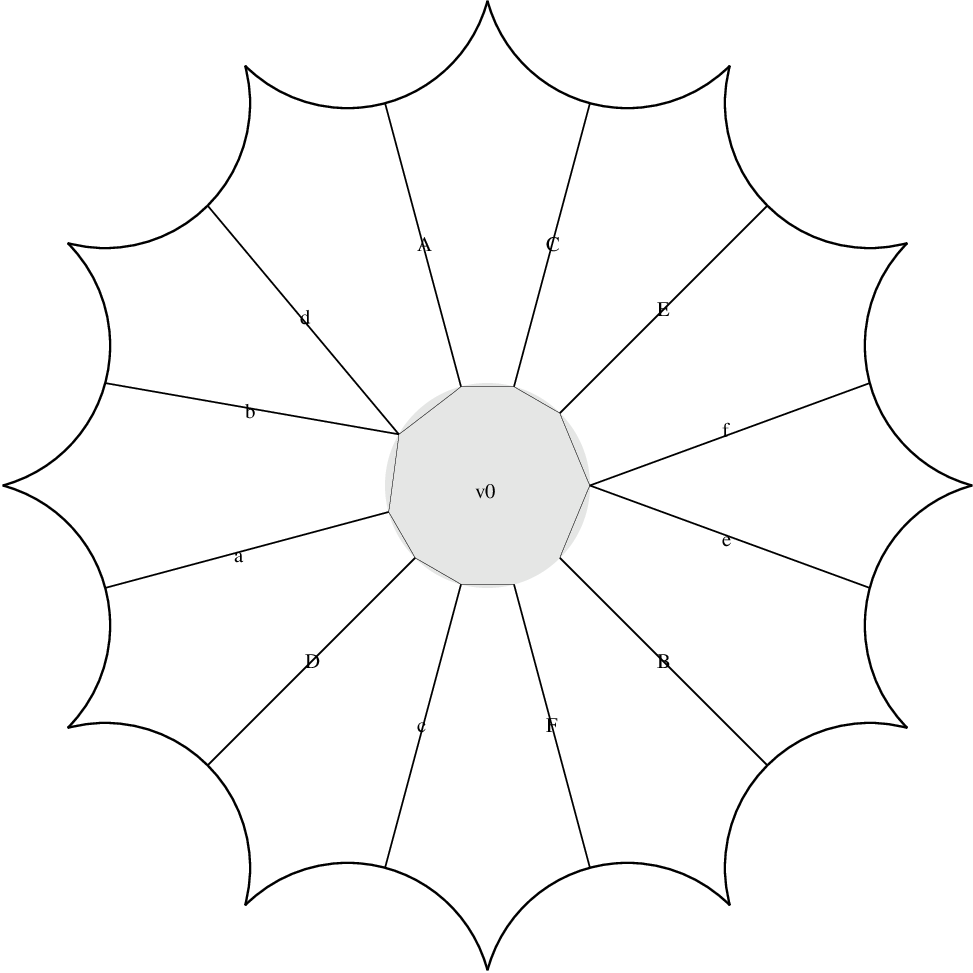}
\label{fig:pennersub2}
}
\end{center}
\caption{Example~\ref{ex:penner}: Penner's pA map, non-orientable train track.}
\label{fig:penner}
\end{figure}
shows five curves $c_1, \ldots, c_5$ on the genus $3$ surface.
The example is the product of positive Dehn twists about $c_1, c_2, c_3$ and inverse twists about $c_4, c_5$.  
The map $f$ acts as follows:
\[
\begin{array}{cclcccl}
a & \mapsto & ada & \qquad &
b & \mapsto & da{\bar c}{\bar d}{\bar a}b\\
c & \mapsto & c{\bar e}{\bar b}adc & \qquad &
d & \mapsto & da{\bar c}{\bar d}{\bar a}befc{\bar e}{\bar b}ad\\
e & \mapsto & efc{\bar e}{\bar b}adada{\bar c}{\bar d}{\bar a}be & \qquad &
f & \mapsto & efc{\bar e}{\bar b}adada{\bar c}{\bar d}{\bar a}bef\\
\end{array}
\]

There is exactly one vertex (which is partial) and the train track is
non-orientable.  The only vertex is fixed by the map and $\chi(f_\ast|_{\im\delta}) = x-1$. 
We have $\dim W(G,f) = 5$, hence the skew-symmetric product is
degenerate and $\dim Z=1$, which means $p(x) = x-1$.
The characteristic polynomial factors as
$\chi(f_*) = (x^4-11x^3+22x^2-11x+1)(x-1)(x-1).$
The symplectic polynomial $s(x)= x^4-11x^3+22x^2-11x+1$ is irreducible, hence  in this example it is necessarily the minimum polynomial of its dilatation.  (Note that this was not the case  in Example~\ref{ex:k89} above.)

If one orients the real and infinitesimal edges of the train track $\tau$ in Figure~\ref{fig:pennersub2} locally so that all orientations are consistent around the single partial vertex $v_0$, one sees that the loops $a,d$ and $f$ do not have globally consistent orientations, whereas the loops $b,c$ and $e$ do.
Since $G$ has no odd vertices, the orientation cover is just an ordinary double cover as illustrated in Figure~\ref{fig:penner-cover}. 
\begin{figure}[htpb!]  
\begin{center}
\input{mixed.psfrag}
\psfrag{a'}{$a'$}
\psfrag{b'}{$b'$}
\psfrag{c'}{$c'$}
\psfrag{d'}{$d'$}
\psfrag{e'}{$e'$}
\psfrag{f'}{$f'$}
\psfrag{1}{$\overline a'$}
\psfrag{2}{$\overline b'$}
\psfrag{3}{$\overline c'$}
\psfrag{4}{$\overline d'$}
\psfrag{5}{$\overline e'$}
\psfrag{6}{$\overline f'$}
\psfrag{0}{$v_0$}
\psfrag{!}{$v_1$}
\includegraphics[width=1\textwidth]{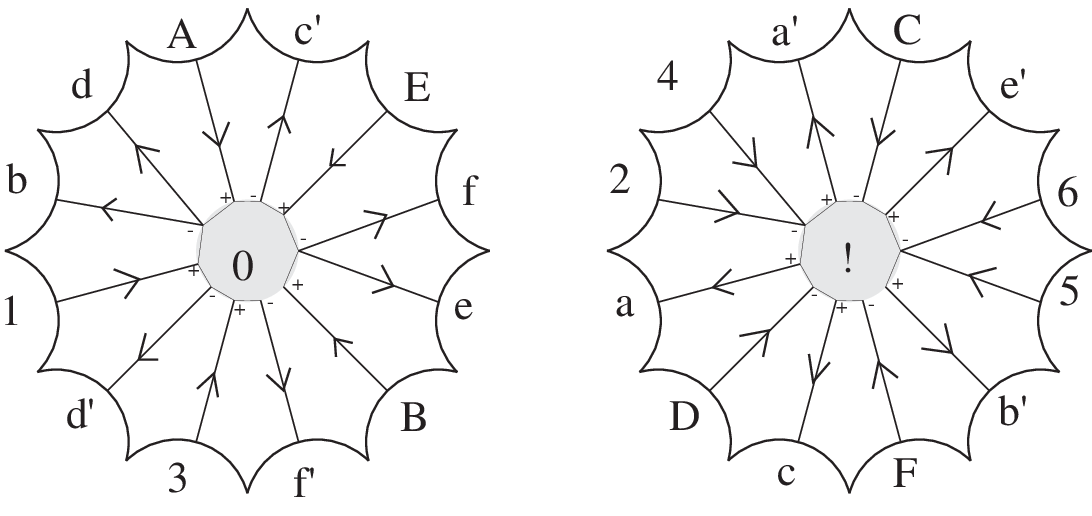}
\end{center}
\caption{ The orientation cover $\tilde \tau$ of $\tau$ in Figure~\ref{fig:penner}.}
\label{fig:penner-cover}
\end{figure}
Each edge, say $a \subset G$, lifts to two copies $a, a' \subset \tilde G$ and the vertex $v_0 \in G$ lifts to $v_0, v_1 \in \tilde G$. 
We choose an orientation for $\tilde \tau$. 
As claimed in Proposition~\ref{prop: orientation cover is orientable}, twin edges have opposite orientations. 
Proposition~\ref{or-cover-matrix} implies that there are two covering maps; orientation preserving and orientation reversing. 
The following is the orientation preserving train track map $\tilde f_{\rm op} : \tilde G \to \tilde G$:
$$
\begin{array}{lcllcl}
a: (v_1, v_0) &\mapsto & ada 
&
a': (v_1, v_0) & \mapsto &  a' d' a' 
\\
b: (v_0, v_0) & \mapsto &  dac'd'a'b
&
b': (v_1, v_1) & \mapsto &  b' a d c  a'   d' 
\\
c: (v_1, v_1) & \mapsto & c e' b' a d c  
&
c' : (v_0, v_0) & \mapsto &  c'd' a' bec' 
\\
d: (v_0, v_1) & \mapsto & dac'  d' a' b   efc e'  b'ad 
&
d': (v_0, v_1) & \mapsto & d' a' bec' f' e'   b'adca'd'
\\
e: (v_0, v_0) & \mapsto &  efce'b'ad  a  dac'd' a' b e 
&
e': (v_1, v_1) & \mapsto & e'  b'adca'd' a' d' a' bec' f' e' 
\\
f: (v_0, v_1) & \mapsto &  efce'b'ad  a  dac'd' a' b e f  
&
f': (v_0, v_1) & \mapsto &   f' e'  b'adca'd' a' d' a' bec' f' e'
\end{array}
$$
We order the edges $a, b, \cdots, f, a', b', \cdots, f'$. Then the transition matrix of $\tilde f_{\rm op}$ has the form 
$\scriptsize
\begin{bmatrix}
A & B\\
B & A
\end{bmatrix}$
where
\begin{equation*}
A=
{\scriptsize
\begin{bmatrix}
2 &  1 &  1 &  2 &  3 &  3 \\ 
0 &  1 &  0 &  1 &  1 &  1 \\
0 &  0 &  2 &  1 &  1 &  1 \\
1 &  1 &  1 &  2 &  2 &  2 \\
0 &  0 &  0 &  1 &  2 &  2 \\
0 &  0 &  0 &  1 &  1 &  2
\end{bmatrix}}
\qquad \mbox{and} \qquad
B=
{\scriptsize
\begin{bmatrix}
0 &  1 &  0 &  1 &  1 &  1\\
0 &  0 &  1 &  1 &  1 &  1\\
0 &  1 &  0 &  1 &  1 &  1\\
0 &  1 &  0 &  1 &  1 &  1\\
0 &  0 &  1 &  1 &  1 &  1\\
0 &  0 &  0 &  0 &  0 &  0
\end{bmatrix}}.
\end{equation*}
Note that $A+B$ is the transition matrix for $f : G\to G$. 
Its characteristic polynomial is
$$(-1 + x)^4 (1 - 4 x + x^2) (1 - 3 x + x^2) (1 - 11 x + 22 x^2 - 
   11 x^3 + x^4),$$
and the symplectic polynomial is 
$$
\chi(\tilde f_{\rm op} {}_\ast |_{W(\tilde G,\tilde f)/\tilde Z})=
(-1 + x)^2 (1 - 4 x + x^2) (1 - 3 x + x^2) (1 - 11 x + 22 x^2 -11 x^3 + x^4).
$$   
The following is the orientation reversing train track map $\tilde f_{\rm or} : \tilde G \to \tilde G$.
$$
\begin{array}{lcllcl}
a: (v_1, v_0) &\mapsto & \overline a'  \overline d'\overline a' 
&
a': (v_1, v_0) & \mapsto &  \overline a \overline d \overline a 
\\
b: (v_0, v_0) & \mapsto &  \overline d' \overline a' \overline c  \overline d \overline a  \overline b' 
&
b': (v_1, v_1) & \mapsto &  \overline b \overline a' \overline d' \overline c'  \overline a \overline d 
\\
c: (v_1, v_1) & \mapsto &  \overline c' \ \overline e \overline b \overline a' \overline d' \overline c' 
&
c': (v_0, v_0) & \mapsto &  \overline c  \overline d \overline a  \overline b' \overline e' \overline c  
\\
d: (v_0, v_1) & \mapsto &  \overline d' \overline a' \overline c  \overline d \overline a  \overline b' \overline e' \overline f'  \overline c' \overline e \overline b  \overline a' \overline d' 
&
d': (v_0, v_1) & \mapsto &  \overline d \overline a  \overline b'  \overline e'  \overline c  \overline f \overline e \overline b  \overline a'  \overline d' \overline c' \overline a \overline d 
\\
e: (v_0, v_0) & \mapsto &  \overline e' \overline f' \overline c'  \overline e \overline b \overline a' \overline d' \overline a' \overline d' \overline a' \overline c \overline d \overline a \overline b' \overline e' 
&
e': (v_1, v_1) & \mapsto &  \overline e \overline b \overline a'  \overline d'  \overline c' \overline a \overline d \overline a \overline d \overline a \overline b'  \overline e'  \overline c  \overline f \overline e 
\\
f: (v_0, v_1) & \mapsto &  \overline e' \overline f' \overline c'  \overline e \overline b \overline a' \overline d' \overline a' \overline d' \overline a' \overline c  \overline d \overline a \overline b' \overline e' \overline f' 
&
f': (v_0, v_1) & \mapsto &   \overline f \overline e \overline b \overline a'  \overline d'  \overline c' \overline a \overline d \overline a \overline d \overline a \overline b'  \overline e'  \overline c   \overline f \overline e 
\end{array}
$$
We observe that the transition matrix for $\tilde f_{\rm or}$ is  
$
\scriptsize
\begin{bmatrix}
B & A\\
A & B
\end{bmatrix}.
$
Its characteristic polynomial is
$$(-1 + x)^2 (1 + x)^2 (1 + 3 x + x^2) (1 + 4 x + x^2) (1 - 11 x + 
   22 x^2 - 11 x^3 + x^4), $$
and the symplectic polynomial is 
$$
\chi(\tilde f_{\rm or} {}_\ast|_{W(\tilde G,\tilde f)/\tilde Z})=
(1 + x)^2 (1 + 3 x + x^2) (1 + 4 x + x^2) (1 - 11 x + 22 x^2 - 11 x^3 + x^4).
$$
The dilatation cannot distinguish the pA maps $F:S\to S$, $\tilde F_{\rm op}: \tilde S\to \tilde S$ and $\tilde F_{\rm or}: \tilde S\to \tilde S$, but our symplectic polynomial can distinguish the three.

It seems to be an open question to describe all the ways to construct all pA maps having a fixed dilatation.  
\end{example}

\begin{example}[Figure~\ref{fig:evenodd}]\label{ex:evenodd}
The following map shows that even, odd, and partial vertices can
coexist in the same (non-orientable) train track:
\begin{figure}[htpb!]  
\begin{center}
\input{mixed.psfrag}
\includegraphics[width=0.5\textwidth]{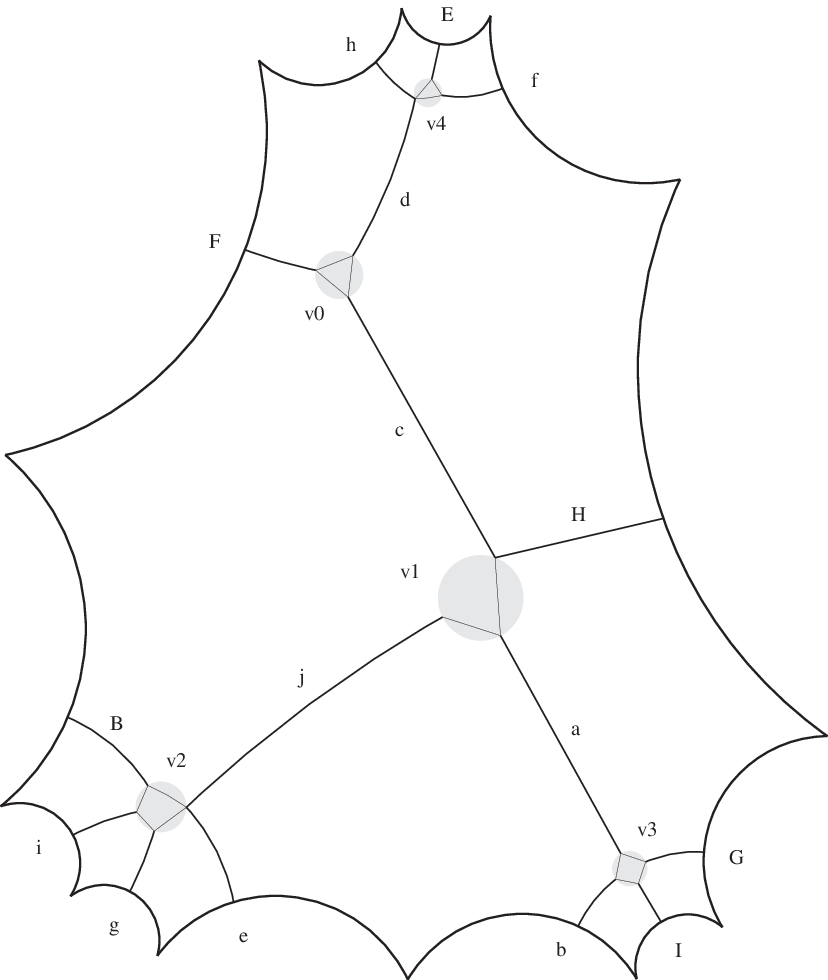}
\end{center}
\caption{Example~\ref{ex:evenodd}: A train track with even, odd, and partial vertices.}
\label{fig:evenodd}
\end{figure}
\[
\begin{array}{lclclclclcl}
a & \mapsto & a{\bar g} & \qquad &
b & \mapsto & ja & \qquad &
c & \mapsto & fc \\
d & \mapsto & ca{\bar g}ibja{\bar g}e & \qquad &
e & \mapsto & {\bar g}eha{\bar g}ibja{\bar g}ef & \qquad &
f & \mapsto & {\bar d} \\
g & \mapsto & {\bar i} & \qquad &
h & \mapsto & c & \qquad &
i & \mapsto & b \\
j & \mapsto & {\bar g}eha{\bar g}ibj
\end{array}
\]

The characteristic polynomial factors as
$$
s(x) p(x) \chi(f_\ast|_{\im\delta})=
(x^6 - 3 x^5 + x^4 - 5 x^3 + x^2 - 3 x + 1) (x-1) (x^3-x^2-x+1).
$$
The factor $\chi(f_\ast|_{\im \delta})=(x^3-x^2-x+1)$ is the characteristic
polynomial of the matrix
$$
\begin{bmatrix}
1&0&0\\
0&0&1\\
0&1&0
\end{bmatrix},
$$
which describes how non-odd vertices $v_1, v_2, v_3$ are permuted by $[F]$. 
\end{example}

\section*{Acknowledgment}
The authors would like to thank Jeffrey Carlson  for his help with the proof of Corollary~\ref{cor:homology polynomial is invariant}, which was stated as a conjecture in an earlier draft. They thank Dan Margalit for suggesting a way to find examples of non-conjugate maps with the same dilatation.  They would also like to thank Mladen Bestvina, Nathan Dunfield, Jordan Ellenberg, Ji-young Ham,  Eriko Hironaka, Eiko Kin, Chris Leininger, Robert Penner, and Jean-Luc Thiffeault for their generosity in sharing their expertise and their patience in responding to questions.   

The third author was partially supported by NSF grants DMS-0806492 and DMS-0635607.


\end{document}